\documentclass[mathpazo]{cicp}

\usepackage{subcaption, bbm, enumitem, color}
\usepackage{epstopdf}

\newcommand{\R}{\mathbb{R}}
\newtheorem{corollary}{Corollary}

\begin{document}
\title[Forces in the $N$-Body Dielectric Spheres Problem]{A Linear Scaling in Accuracy Numerical Method for Computing the Electrostatic Forces in the $N$-Body \\Dielectric Spheres Problem}


 \author[M. Hassan and B. Stamm]{Muhammad Hassan\affil{1}\comma\corrauth,
       and Benjamin Stamm\affil{1}}
 \address{\affilnum{1}\ Center for Computational Engineering Science,
          Department of Mathematics, RWTH Aachen University, Schinkelstrasse 2, 52062 Aachen, Germany.}
 \emails{{\tt hassan@mathcces.rwth-aachen.de} (M.~Hassan),
          {\tt stamm@mathcces.rwth-aachen.de} (B.~Stamm)}

\begin{abstract}
This article deals with the efficient and accurate computation of the electrostatic forces between charged, spherical dielectric particles undergoing mutual polarisation. We use the spectral Galerkin boundary integral equation framework developed by Lindgren et al. (J. Comput. Phys. 371 (2018): 712-731) and subsequently analysed in two earlier contributions of the authors to propose a linear scaling in cost algorithm for the computation of the approximate forces. We establish exponential convergence of the method and derive error estimates for the approximate forces that do not explicitly depend on the number of dielectric particles $N$. Consequently, the proposed method requires only $\mathcal{O}(N)$ operations to compute the electrostatic forces acting on $N$ dielectric particles up to any given and fixed relative error.
\end{abstract}


\maketitle

\section{Introduction}
Predicting the motion of a large number of objects interacting under the influence of a potential field, commonly known as the $N$-body problem, is one of the most well-known problems of classical physics. The $N$-body problem first arose due to the desire of astronomers to explain the motion of celestial objects interacting due to gravity \cite{newton1934principia,poincare} but the problem is also ubiquitous in physical phenomena involving large-scale electrostatic interactions. Thus, understanding the behaviour of charged colloidal particles (see, e.g., \cite{barros2014dielectric,dobnikar2002many,hynninen2006prediction,Extra1,merrill2009many,image1}) or the fabrication of binary nanoparticle superlattices and so-called Coulombic crystals (see, e.g., \cite{bartlett2005three,boles2015many,Crystal,kostiainen2013electrostatic,Self,ristenpart2003electrically,Lattice}), or the assembly of proteins and other cellular structures (see, e.g., \cite{PhysRevE.73.061902,ejtehadi2004three,patel2004charmm,sagui1999molecular,van2000many,yap2013calculating}) all require knowledge of Coulomb interactions between a large number of physical objects.

Many such electrostatic phenomena involve interactions between charged, spherical dielectric particles embedded in a dielectric medium, undergoing mutual polarisation. One is then typically interested in either the total electrostatic energy of the system or the electrostatic force acting on each particle, both of which can be derived from knowledge of the electric potential generated by these particles. Knowledge of the forces in particular is required if one wishes to perform molecular dynamics simulations or study assembly processes of charged particles (see, e.g., \cite{dahirel2007new,li2013self,prochazka2016self,zhang2010cluster} as well as the references on superlattices given above). In contrast to the much simpler case of point-charges however, a full description of the electric potential generated by such polarisable particles cannot be obtained as simply the sum of pairwise interactions. Instead, the potential is realised as the solution to a PDE, posed on the full three-dimensional space with interface conditions on the boundaries of the spherical particles (see, e.g., \cite{boles2015many,Hassan1,lindgren2018}). Since this PDE cannot generally be solved analytically, it becomes necessary to use some numerical method to first compute the \emph{approximate} electric potential and then use this to obtain \emph{approximations} to either the total energy or the force acting on each particle. It is therefore of great interest to develop efficient numerical algorithms that can yield approximations to the energy and the forces with theoretically quantifiable error estimates.

A number of different approaches to this so-called \emph{$N$-body dielectric spheres electrostatic interaction} problem have been proposed in the literature (see, e.g., \cite{barros2014efficient,multipole1,freed2014perturbative,multipole2,lotan2006,image1,image3,image2}). Unfortunately, many of these methods suffer from the handicap that they may become computationally prohibitive if the number of particles is very large. Additionally, these method have typically been formulated in a manner that makes them unsuitable for a systematic numerical analysis. As a consequence, it is usually not possible to theoretically evaluate the accuracy of these methods and, in particular, to explore the dependence of the accuracy on the number of dielectric particles $N$. These drawbacks are particularly regrettable since the quality of an $N$-body numerical method is assessed precisely by considering how the accuracy and computational cost of the algorithm scale with $N$. Indeed, {  given a family of geometrical configurations with varying number of dielectric spheres $N$}, using the terminology stated in \cite{Hassan2}:
\begin{itemize}
	\item We say that an $N$-body numerical method is \emph{$N$-error stable} if, for a fixed number of degrees of freedom per object, the relative or average error in the approximate solution for different geometrical configurations does not increase with $N$. 
	
	\item We say that an $N$-body numerical method is \emph{linear scaling in cost} if, given a geometrical configuration with $N$ spheres and for a fixed number of degrees of freedom per object, the numerical method requires $\mathcal{O}(N)$ operations to compute an approximate solution with a given and fixed tolerance. %
	
	\item Finally, we say that an $N$-body numerical method is \emph{linear scaling in accuracy} if it is both $N$-error stable and linear scaling in cost. 
\end{itemize}

Linear scaling in accuracy methods can be viewed as the gold-standard for $N$-body problems since these methods require only $\mathcal{O}(N)$ operations to compute an approximate solution with a given average error (the total error scaled by $N$) or relative error. Note that achieving the required linear scaling in \emph{cost} typically requires the use of fast summation methods such as tree codes (see, e.g., \cite{appel1985efficient,barnes1986hierarchical,boateng2013comparison,dehnen2000very,li2009cartesian}) including the so-called Fast Multipole method (see \cite{cheng1999fast,greengard1990numerical,greengard1}, or particle mesh and P3M methods (see, e.g., \cite{efstathiou1985numerical,hockney,knebe2001multi}). Additionally, it must be shown that the number of solver iterations required to obtain an approximate solution for different geometrical configurations does not grow with~$N$. 

E. Lindgren and coworkers recently proposed in \cite{lindgren2018}, a computational method based on a spectral Galerkin discretisation of a second-kind integral equation posed on the boundaries of the spherical dielectric particles. The boundary integral equation (BIE) was formulated in terms of the so-called induced surface charge on each spherical particle, which could be used to deduce physical quantities of interest. Indeed, through the use of the FMM, the method was empirically shown to achieve linear scaling in cost for the computation of the total electrostatic energy. The practical utility of this new algorithm was, for instance, demonstrated in the contributions \cite{Titan,lindgren2018dynamic}. Furthermore, in the articles \cite{Hassan1} and \cite{Hassan2}, the authors presented a complete numerical and complexity analysis of the method and rigorously established that, for families of geometrical configurations satisfying appropriate assumptions, the method is linear scaling in accuracy for the computation of the induced surface charge and total energy.

In this work, we propose and systematically analyse an efficient numerical method, based on the Galerkin BIE framework of Lindgren et al. and using the FMM, for the computation of the electrostatic forces acting on charged, dielectric spherical particles embedded in a homogenous dielectric medium undergoing mutual polarisation. The proposed algorithm can handle an arbitrary number of spherical particles of varying dielectric constants and radii, thus making it a powerful tool for practical applications. Furthermore, inspired by the analysis in the previous works \cite{Hassan1,Hassan2}, we prove that, under suitable geometrical assumptions which we describe in detail later, this numerical method achieves \underline{linear scaling in accuracy} for the computation of the forces. In order to demonstrate this claim, we first derive convergence rates for the approximate electrostatic forces that are explicitly independent of $N$, and we then present a linear scaling in cost solution strategy for the computation of the approximate forces. As a corollary of the convergence rates, we also obtain exponential convergence of the forces under appropriate regularity assumptions. Numerical evidence is provided that supports our theoretical results.

The remainder of this article is organised as follows. In Section 2, we introduce notation, describe the problem setting and the governing boundary integral equation, and restate the main tools and results from the papers \cite{Hassan1} and \cite{Hassan2} that we require for the subsequent analysis. In Section 3 we describe two approaches-- motivated by different physical considerations-- to defining the electrostatic forces, and we show that these are equivalent. In Section 4, we derive $N$- independent convergence rates for the forces thereby establishing $N$-error stability of our numerical method. In Section 5, we state a linear scaling in cost solution strategy using the FMM and present numerical results supporting our theoretical claims. Finally, in Section 6, we present our conclusion and discuss possible extensions.

\section{Problem Setting and Previous Results}\label{sec:2}
Throughout this article, we will use well-known results and notation from the theory of boundary integral equations. Most of these definitions and results can be found in standard textbooks such as \cite{McLean} or \cite{Schwab}. Additionally, we will state some results from the articles \cite{Hassan1} and \cite{Hassan2} that we will require for our subsequent analysis.

\subsection{Setting and Basic Notions}\label{sec:2a}
{We begin by describing precisely the types of geometrical situations we will consider in this article. As indicated in the introduction, we are interested in studying geometrical configurations that are the unions of an arbitrary number $N$ of non-intersecting open balls with varying radii in three dimensions. As in the previous contribution \cite{Hassan1} however, our claim of $N$-independent error estimates requires us to impose certain assumptions on the types of geometries we consider. To this end, let $\mathcal{I}$ denote a countable indexing set. We consider a so-called family of geometries~$\{\Omega_{\mathcal{F}}\}_{\mathcal{F} \in \mathcal{I}}$. Each element $\Omega_{\mathcal{F}} \subset \mathbb{R}^3$ in this family is the (set) union of a fixed number of non-intersecting open balls of varying locations and radii with associated dielectric constants, and therefore represents a particular physical geometric situation. It is easy to see that each element $\Omega_{\mathcal{F}}$ of this family of geometries is uniquely determined by the following four parameters:
	\begin{itemize}
		\item A non-zero number $N_{\mathcal{F}} \in \mathbb{N}$, which represents the total number of dielectric spherical particles that compose the geometry $\Omega_{\mathcal{F}}$;
		\item A collection of points $\{\bold{x}^{\mathcal{F}}_i\}_{i=1}^{N_{\mathcal{F}}} \in \mathbb{R}^3$, which represent the centres of the spherical particles composing the geometry $\Omega_{\mathcal{F}}$;
		\item A collection of positive real numbers $\{r_i^{\mathcal{F}}\}_{i=1}^{N_{\mathcal{F}}} \in \mathbb{R}$, which represent the radii of the spherical particles composing the geometry $\Omega_{\mathcal{F}}$;
		\item A collection of positive real numbers $\{\kappa^{\mathcal{F}}_i\}_{i=0}^N \in \mathbb{R}$. Here, $\kappa^{\mathcal{F}}_0$ denotes the dielectric constant of the external medium while $\{\kappa^{\mathcal{F}}_i\}_{i=1}^N$ represent the dielectric constants of each dielectric sphere.
	\end{itemize}
	Indeed, using the first three parameters we can define the open balls $\Omega^\mathcal{F}_i:= \mathcal{B}_{r_i}(\bold{x}_i) \subset \mathbb{R}^3$, $i \in \{1, \ldots, N_{\mathcal{F}}\}$ which represent the spherical dielectric particles composing the geometry $\Omega_{\mathcal{F}}$, i.e., $\Omega_{\mathcal{F}}= \cup_{i=1}^{N_\mathcal{F}} \Omega_i^{\mathcal{F}}$. Moreover, the fourth parameter $\{\kappa^{\mathcal{F}}_i\}_{i=0}^N$ denotes the dielectric constants associated with this geometry.

Exactly as in \cite{Hassan1}, we now impose the following three important assumptions on the above parameters:
	
	\begin{enumerate}
		
		\item[\textbf{A1:}] \textbf{[Uniformly bounded radii]} There exist constants $r^{\infty}_->0$ and $r^{\infty}_+>0$ such that 
		\begin{align*}
		\inf_{\mathcal{F} \in \mathcal{I}}\, \min_{i=1, \ldots, N_{\mathcal{F}}} r^{\mathcal{F}}_i > r^{\infty}_- \quad \text{and} \quad \sup_{\mathcal{F} \in \mathcal{I}}\, \max_{i=1, \ldots, N_{\mathcal{F}}} r^{\mathcal{F}}_i < r^{\infty}_+.
		\end{align*}
		
		\item[\textbf{A2:}] \textbf{[Uniformly bounded minimal separation]} There exists a constant $\epsilon^{\infty} > 0$ such that 
		\begin{align*}
		\inf_{\mathcal{F} \in \mathcal{I}}\, \min_{\substack{i, j=1, \ldots, N_{\mathcal{F} } \\ i \neq j}} \big(\vert \bold{x}_i^{\mathcal{F}} -\bold{x}_j^{\mathcal{F}}\vert - r^{\mathcal{F}}_i -r^{\mathcal{F}}_j\big)> \epsilon^{\infty}.
		\end{align*}
		
		\item[\textbf{A3:}] \textbf{[Uniformly bounded dielectric constants]} There exist constants $\kappa^{\infty}_->0$ and $\kappa^{\infty}_+>0$ such that 
		\begin{align*}
		\inf_{\mathcal{F} \in \mathcal{I}} \,\min_{i=1, \ldots, N_{\mathcal{F}}} \kappa^{\mathcal{F}}> \kappa^{\infty}_- \quad \text{and} \quad \sup_{\mathcal{F} \in \mathcal{I}} \, \max_{i=1, \ldots, N_{\mathcal{F}}} \kappa^{\mathcal{F}} < \kappa^{\infty}_+.
		\end{align*}
	\end{enumerate}
	
	\noindent	In other words we assume that the family of geometries $\{\Omega_{\mathcal{F}}\}_{\mathcal{F} \in \mathcal{I}}$ we consider in this article describe physical situations where the radii of the dielectric spherical particles, the minimum inter-sphere separation distance and the dielectric constants are all uniformly bounded. These assumptions are necessary because the error estimates we will derive, while explicitly independent of the number of dielectric particles $N_\mathcal{F}$, do depend on other geometrical parameters, and we would thus like to avoid situations where these geometric parameters degrade with increasing $N_\mathcal{F}$. Since many physical situations involve non-metallic dielectric particles which neither have vanishing or exploding dielectric constants nor vanishing or exploding radii (see, e.g., \cite{Crystal,lee2015direct,lindgren2018dynamic,Self,soh2014charging}), these assumptions do not greatly limit the scope of our results.
	
	In the remainder of this article, we will consider a fixed geometry from the family of geometries $\{\Omega_{\mathcal{F}}\}_{\mathcal{F} \in \mathcal{I}}$ satisfying the assumptions \textbf{A1)-A3)}. To avoid bulky notation we will drop the superscript and subscript $\mathcal{F}$ and denote this geometry by $\Omega^-$. The geometry is constructed as follows: Let $N \in \mathbb{N}$, let $\{\bold{x}_i\}_{i=1}^N \in \mathbb{R}^3$ be a collection of points in $\mathbb{R}^3$, and let $\{r_i\}_{i=1}^N \in \mathbb{R}$ be a collection of positive real numbers. For each $i \in \{1, \ldots, N\}$ we define $\Omega_i := \mathcal{B}_{r_i}(\bold{x}_i) \subset \mathbb{R}^3$ as the open ball of radius $r_i >0$ centred at the point $\bold{x}_i$. $\Omega^- \subset \mathbb{R}^3$ is defined as $\Omega^-:= \cup_{i=1}^N \Omega_i$. Furthermore, we define $\Omega^+:= \mathbb{R}^3 \setminus \overline{\Omega^-}$, and we write $\partial \Omega$ for the boundary of $\Omega^-$ and {$\eta(\bold{x})$ for the unit normal vector at $\bold{x} \in \partial \Omega$ pointing towards the exterior of~$\Omega^-$}.} Additionally, we denote by $\{\kappa_i\}_{i=1}^N \in \mathbb{R}_+$ the dielectric constants of all spherical particles $\{\Omega_i\}_{i=1}^N$ and by $\kappa_0 \in \mathbb{R}_+$ the dielectric constant of the background medium. To aid our exposition, we define the function $\kappa \colon \partial \Omega \rightarrow \mathbb{R}$ as $\kappa (\bold{x}):= \kappa_i ~\text{ for } \bold{x} \in \partial \Omega_i.$ 

Next, we define the Sobolev space $H^1(\Omega^-):= \left\{u \in L^2(\Omega^-) \colon \nabla u \in L^2(\Omega^-)\right\}$ with norm $\Vert u \Vert^2_{H^1(\Omega^-)}$ $:=\sum_{i=1}^N\Vert u \Vert^2_{L^2(\Omega_i)} + \Vert \nabla u \Vert^2_{L^2(\Omega_i)}$. We further define the weighted Sobolev space $H^1(\Omega^+)$ as the completion of $C^{\infty}_{\text{comp}}(\Omega^+)$ with respect to the norm $\Vert u \Vert^2_{H^1(\Omega^+)}:= \int_{\Omega^+} \frac{\vert v(\bold{x})\vert^2}{1 + \vert \bold{x}\vert^2}\, d\bold{x}+ \Vert \nabla v\Vert^2_{L^2(\Omega^+)}$. Functions that satisfy the decay conditions associated with exterior Laplace problems belong to this space (see, e.g., \cite[Section 2.9.2.4]{Schwab}). Additionally, we denote by $H^{\frac{1}{2}}(\partial \Omega)$ the Sobolev space of order $\frac{1}{2}$ with the Sobolev-Slobodeckij norm $\Vert \lambda \Vert^2_{H^{\frac{1}{2}}(\partial \Omega)}:=\sum_{i=1}^N \Vert \lambda\Vert^2_{L^2(\partial \Omega_i)} + \int_{\partial \Omega_i} \int_{\partial \Omega_i} \frac{\vert\lambda(\bold{x})-\lambda(\bold{y})\vert^2}{\vert \bold{x} - \bold{y}\vert^3 } \, d\bold{x} d\bold{y}$. Notice that we have chosen to define $\Vert \cdot \Vert^2_{H^{\frac{1}{2}}(\partial \Omega)}$ as a sum of local norms on each sphere. We also define $H^{-\frac{1}{2}}(\partial \Omega):=\left(H^{\frac{1}{2}}(\partial \Omega)\right)^*$, and we equip this dual space with the canonical dual norm $\Vert \cdot \Vert_{H^{-\frac{1}{2}}(\partial \Omega)}$.

For the sake of brevity, when there is no possibility of confusion, we will use the notation $\langle \cdot, \cdot \rangle_{\partial \Omega}$ and $\langle \cdot, \cdot \rangle_{\partial \Omega_i}$ to denote the duality pairings $\langle \cdot, \cdot \rangle_{H^{-\frac{1}{2}}(\partial \Omega) \times H^{\frac{1}{2}}(\partial \Omega)}$ and \\ $\langle \cdot, \cdot \rangle_{H^{-\frac{1}{2}}(\partial \Omega_i) \times H^{\frac{1}{2}}(\partial \Omega_i)}$ for a given $i \in \{1, \ldots, N\}$ respectively.

Equipped with these function spaces, we can introduce the fundamental linear operators we require for the subsequent exposition. We first introduce the mappings $\gamma^- \colon H^1(\Omega^{-})$ $\rightarrow H^{\frac{1}{2}}(\partial \Omega)$ and $\gamma^+ \colon H^1(\Omega^+) \rightarrow H^{\frac{1}{2}}(\partial \Omega)$ as the continuous, linear and surjective interior and exterior Dirichlet trace operators respectively (see, for example, \cite[Theorem 2.6.8, Theorem 2.6.11]{Schwab}). Next, for each $s \in \{+, -\}$ we define the closed subspace $\mathbb{H}(\Omega^s):= 
\{u \in H^{1}(\Omega^s) \colon \Delta u =0 \text{ in } \Omega^s\},$
and we write $\gamma^-_N \colon \mathbb{H}(\Omega^-) \rightarrow H^{-\frac{1}{2}}(\partial \Omega)$ and $\gamma^+_N \colon \mathbb{H}(\Omega^+) \rightarrow H^{-\frac{1}{2}}(\partial \Omega)$ for the interior and exterior Neumann trace operator respectively {(see \cite[Theorem 2.8.3]{Schwab} for precise conventions)}. We remark that the interior and exterior Dirichlet and Neumann trace operators can be defined analogously for functions of appropriate regularity defined on $\Omega^- \cup \Omega^+$ or $\mathbb{R}^3$. In addition, we introduce the so-called (interior) Dirichlet-to-Neumann map $\text{DtN} \colon H^{\frac{1}{2}}(\partial \Omega) \rightarrow H^{-\frac{1}{2}}(\partial \Omega)$ as follows: Given a function $\lambda \in H^{\frac{1}{2}}(\partial \Omega)$, DtN$\lambda = \gamma_N^- u_{\lambda} \in H^{-\frac{1}{2}}(\partial \Omega)$, where $u_{\lambda} \in H^1(\Omega^-)$ is the unique harmonic function such that $\gamma^- u_{\lambda}=\lambda$. Finally, for each $\nu \in H^{-\frac{1}{2}}(\partial \Omega)$ and all $\bold{x} \in \mathbb{R}^3 \setminus \partial \Omega$ we define the function
\begin{align*}
\mathcal{S}(\nu)(\bold{x})&:=\int_{\partial \Omega} \frac{\nu(\bold{y})}{4\pi\vert \bold{x}- \bold{y}\vert}\, d \bold{y},
\end{align*}
where the integral is understood as a $\langle \cdot, \cdot \rangle_{H^{-\frac{1}{2}}(\partial \Omega) \times H^{\frac{1}{2}}(\partial \Omega) }$ duality pairing. The mapping $\mathcal{S}$ is known as the single layer potential. It can be shown (see, e.g., \cite[Chapter 2]{Schwab}) that $\mathcal{S}$ is a linear bounded operator from $H^{-\frac{1}{2}}(\partial \Omega)$ to $H^{1}_{\rm loc}\left(\mathbb{R}^3\right)$, and also that $\mathcal{S}$ maps into the space of harmonic functions on the complement $\mathbb{R}^3 \setminus \partial \Omega$ of the boundary. The single layer potential allows us to define the boundary integral operator
\begin{align*}
\mathcal{V}&:= \big(\gamma^- \circ \mathcal{S}\big) \hspace{0.0cm}\colon \hspace{0.0cm}H^{-\frac{1}{2}}(\partial \Omega) \rightarrow H^{\frac{1}{2}}(\partial \Omega).
\end{align*}
The mapping $\mathcal{V}$ is also a bounded linear operator and is called the single layer boundary operator. Detailed definitions and properties of $\mathcal{V}$ as well as other boundary integral operators can, for instance, be found in \cite[Chapters 6, 7]{McLean} or \cite[Chapter 3]{Schwab}. We state some basic properties of $\mathcal{V}$ that we require for our analysis.

{ 
	\begin{lemma}[Properties of $\mathcal{V}$]\label{lem:single}~
		The single layer boundary operator $\mathcal{V} \colon H^{-\frac{1}{2}}(\partial \Omega)$ $\rightarrow H^{\frac{1}{2}}(\partial \Omega)$ is Hermitian and coercive, i.e., there exists a constant $c_{\mathcal{V}} > 0$ that depends on the radii of the open balls and the minimum inter-sphere separation distance but is independent of $N$ such that for all  $\sigma \in H^{-\frac{1}{2}}(\partial \Omega)$ it holds that
		\begin{align*}
		\langle \sigma, \mathcal{V}\sigma \rangle_{\partial \Omega} \geq c_{\mathcal{V}} \Vert \sigma\Vert^2_{H^{-\frac{1}{2}}(\partial \Omega)}.
		\end{align*}
		
		Consequently, the inverse $\mathcal{V}^{-1} \colon H^{\frac{1}{2}}(\partial \Omega) \rightarrow H^{-\frac{1}{2}}(\partial \Omega)$ is also a Hermitian, coercive and bounded linear operator. Additionally, for all $\lambda \in H^{\frac{1}{2}}(\partial \Omega)$, we have the bound
		\begin{align}\label{eq:contraction}
		\langle \text{\rm DtN}\lambda,  \mathcal{V}\text{\rm DtN} \lambda\rangle_{\partial \Omega} \leq \langle \mathcal{V}^{-1}\lambda,\lambda \rangle_{\partial \Omega} \leq \frac{1}{c_{\mathcal{V}}}
		\Vert \lambda \Vert^2_{H^{\frac{1}{2}}(\partial \Omega)}.
		\end{align}
	\end{lemma}
	\begin{proof}
		The fact that $\mathcal{V}$ is Hermitian and coercive with Hermitian, coercive and bounded inverse is a classical result and can, for instance, be found in \cite[Chapter 7]{McLean}. The precise dependencies of the coercivity constant $c_{\mathcal{V}}$ were established in \cite[Lemma 4.7, 4.8]{Hassan1}. The bound \eqref{eq:contraction} is proven in \cite{steinbach2001c}. 
\end{proof}}

Lemma \ref{lem:single} implies in particular that $\mathcal{V}$ induces a norm $\Vert \cdot \Vert_{\mathcal{V}}$ on $H^{-\frac{1}{2}}(\partial \Omega)$ and the inverse $\mathcal{V}^{-1}$ induces a norm $\Vert \cdot \Vert_{\mathcal{V}^{-1}}$ on $H^{\frac{1}{2}}(\partial \Omega)$.

\subsection{Abstract Electrostatic Interaction Problem for Dielectric Spheres}\label{sec:2b}

We now state the problem we wish to analyse. To avoid trivial situations, we assume throughout this article that the spherical particles and the background medium have different dielectric constants, i.e., $\kappa_j\neq \kappa_0 ~\forall j=1, \ldots, N$ (for more details, see \cite[Remark 2.5]{Hassan1}). \vspace{0.5cm}

\noindent{\textbf{Boundary Integral Equation for the Induced Surface Charge}}~

Let $\sigma_f \in H^{-\frac{1}{2}}(\partial \Omega)$. Find $\nu \in H^{-\frac{1}{2}}(\partial \Omega)$ with the property that
\begin{align}\label{eq:3.3a}
\nu - \frac{\kappa_0-\kappa}{\kappa_0} (\text{DtN}\mathcal{V})\nu= \frac{4\pi}{\kappa_0}\sigma_f.
\end{align}
Here, the function $\sigma_f \in H^{-\frac{1}{2}}(\partial \Omega)$ is called the free charge and is a priori known. From a physical point of view, this is the charge distribution (up to some scaling) on each dielectric particle in the absence of any polarisation effects, i.e., if $\kappa=\kappa_0$. The unknown function $\nu \in H^{-\frac{1}{2}}(\partial \Omega)$ is known as the induced surface charge and represents, physically, the charge distribution that results on each dielectric sphere after including polarisation effects.

\begin{remark}
	The boundary integral equation (BIE) \eqref{eq:3.3a} can be derived from a PDE-based transmission problem as discussed in \cite{Hassan1} and \cite{lindgren2018}. Additionally, it is shown in \cite{Hassan1} that the BIE \eqref{eq:3.3a} can be reformulated as a boundary integral equation of the second kind, i.e., as an equation involving a compact perturbation of the identity.
\end{remark}

The BIE \eqref{eq:3.3a} describes the electrostatic interaction between dielectric spheres undergoing mutual polarisation in terms of the induced surface charge on each particle. An alternative approach is to consider, as the quantity of interest, the so-called surface electrostatic potential $\lambda:= \mathcal{V}\nu \in H^{\frac{1}{2}}(\partial \Omega)$ (see \cite{Hassan1,Hassan2} for a detailed exposition) but we do not consider this approach here. Instead, we focus on the Galerkin discretisation of the BIE \eqref{eq:3.3a}. To this end, we first introduce the approximation space. In the sequel, we denote by $\mathbb{N}_0$ the set of non-negative integers.


\begin{definition}[Spherical Harmonics]
	Let $\ell \in \mathbb{N}_0$ and $m \in \{-\ell, \ldots, \ell\}$ be integers. Then we denote by $\mathcal{Y}_\ell^m$ the real-valued $L^2$-orthonormal spherical harmonic of degree $\ell$ and order $m$. A precise definition can be found in \cite{Hassan1}.
\end{definition}

\begin{definition}[Approximation Space on a Sphere]\label{def:6.6}
	Let $\mathcal{O}_{\bold{x}_0} \subset \mathbb{R}^3$ be an open ball of radius $r > 0$ centred at the point $\bold{x}_0 \in \mathbb{R}^3$ and let $\ell_{\max} \in \mathbb{N}_0$. We define the finite-dimensional Hilbert space $W^{\ell_{\max}}(\partial \mathcal{O}_{\bold{x}_0}) \subset {H}^{\frac{1}{2}}(\partial\mathcal{O}_{\bold{x}_0}) \subset  {H}^{-\frac{1}{2}}(\partial\mathcal{O}_{\bold{x}_0})$ as the vector space
	\begin{align*}
	W^{\ell_{\max}}(\partial \mathcal{O}_{\bold{x}_0}):= \Big\{u \colon \partial \mathcal{O}_{\bold{x}_0} \rightarrow \mathbb{R} &\text{ such that } u(\bold{x})= \sum_{{\ell}=0}^{\ell_{\max}} \sum_{m=-\ell}^{m=+\ell} [u]^\ell_m \mathcal{Y}_{\ell}^m\left(\frac{\bold{x}-\bold{x}_0}{\vert \bold{x}-\bold{x}_0\vert}\right)\text{ where }[u]_{\ell}^m \in \mathbb{R}\Big\},
	\end{align*}
	equipped with the inner product
	\begin{align}\label{eq:SH1}
	(u, v)_{W^{\ell_{\max}}(\partial \mathcal{O}_{\bold{x}_0})}:=r^2 [u]_0^0 [v]_0^0 + r^2 \sum_{\ell=1}^{\ell_{\max}} \sum_{m=-\ell}^{m=+\ell} \frac{\ell}{r} [u]_{\ell}^m [v]_{\ell}^m \qquad \forall u, v \in W^{\ell_{\max}}(\partial \mathcal{O}_{\bold{x}_0}).
	\end{align}	
\end{definition}

It is now straightforward to extend the Hilbert space defined in Definition \ref{def:6.6} to the domain $\partial \Omega$.

\begin{definition}[Global Approximation Space]\label{def:6.7}
	Let $\ell_{\max}\in \mathbb{N}_0$. We define the finite-dimensional Hilbert space $W^{\ell_{\max}} \subset H^{\frac{1}{2}}(\partial \Omega) \subset H^{-\frac{1}{2}}(\partial \Omega)$ as the vector space
	\begin{align*}
	W^{\ell_{\max}}:= \Big\{u \colon \partial\Omega \rightarrow \mathbb{R} \text{ such that } \forall i \in \{1, \ldots, N\} \colon u\vert_{\partial \Omega_i} \in W^{\ell_{\max}}(\partial \Omega_i)\Big\},
	\end{align*}
	equipped with the inner product	$(u, v)_{W^{\ell_{\max}}}:= \sum_{i=1}^N \left(u, v\right)_{W^{\ell_{\max}}(\partial \Omega_i)} \forall u, v \in W^{\ell_{\max}}$.
\end{definition}

\noindent{\textbf{Galerkin Discretisation of the Integral Equation \eqref{eq:3.3a}}}~

Let $\sigma_f \in {H}^{-\frac{1}{2}}(\partial \Omega)$ and let $\ell_{\max} \in \mathbb{N}_0$. Find $\nu_{\ell_{\max}} \in W^{\ell_{\max}}$ such that for all $\psi_{\ell_{\max}} \in W^{\ell_{\max}}$ it holds that
\begin{align}\label{eq:Galerkina}
\left(\nu_{\ell_{\max}} - \frac{\kappa_0-\kappa}{\kappa_0}\left(\text{\rm DtN}\mathcal{V}\right) \nu_{\ell_{\max}} , \psi_{\ell_{\max}}\right)_{L^2(\partial \Omega)}= \frac{4\pi}{\kappa_0}\big(\sigma_f, \psi_{\ell_{\max}}\big)_{L^2(\partial \Omega)}.
\end{align}

The boundary integral equation \eqref{eq:3.3a} and the Galerkin discretisation \eqref{eq:Galerkina} have been analysed in the contributions \cite{Hassan1} and \cite{Hassan2}, and we refer interested readers to these papers for a detailed exposition on this topic. However, the basic framework developed in these contributions will be of use in our analysis and we therefore recall some of the key tools and results from these papers.

\subsection{Analysis Framework}
\begin{definition} 
	We define the $N$-dimensional, closed subspace $\mathcal{C}(\partial \Omega) \subset H^{\frac{1}{2}}(\partial \Omega)$ as
	\begin{align*}
	\mathcal{C}(\partial \Omega):= \left\{u \colon \partial \Omega \rightarrow \mathbb{R} \colon \forall i =1, \ldots, N \text{ the restriction }u|_{\partial \Omega_i} \text{ is a constant function}\right\}.
	\end{align*}
	Additionally, we define the closed subspaces $\breve{H}^{\frac{1}{2}}(\partial \Omega)$ $\subset H^{\frac{1}{2}}(\partial \Omega)$ and $\breve{H}^{-\frac{1}{2}}(\partial \Omega)\subset \breve{H}^{\frac{1}{2}}(\partial \Omega)$.
	\begin{align*}
	\breve{H}^{\frac{1}{2}}(\partial \Omega)&:=\left\{u \in H^{\frac{1}{2}}(\partial \Omega) \colon (u, v)_{L^2(\partial \Omega)}=0 \hspace{7mm}\forall v \in \mathcal{C}(\partial \Omega)\right\},\\
	\breve{H}^{-\frac{1}{2}}(\partial \Omega)&:=\left\{\phi \in H^{-\frac{1}{2}}(\partial \Omega) \colon \langle \phi, v\rangle_{\partial \Omega}=0 ~\qquad ~\forall v \in \mathcal{C}(\partial \Omega)\right\}.
	\end{align*}
\end{definition}

The following result is simple to establish.
\begin{lemma}\label{def:projections} 
	There exist complementary decompositions (in the sense of Brezis \cite[Section 2.4]{Brezis}) of the spaces $H^{\frac{1}{2}}(\partial \Omega)$ and $H^{-\frac{1}{2}}(\partial \Omega)$ given by
	\begin{align}\label{eq:Appendix1}
	H^{\frac{1}{2}}(\partial \Omega)=\breve{H}^{\frac{1}{2}}(\partial \Omega) \oplus \mathcal{C}(\partial \Omega)\quad \text{and} \quad
	H^{-\frac{1}{2}}(\partial \Omega)=\breve{H}^{-\frac{1}{2}}(\partial \Omega) \oplus \mathcal{C}(\partial \Omega).
	\end{align}
	
	{     Moreover, the projection operators $\mathbb{P}^{\perp}_0\colon H^{\frac{1}{2}}(\partial \Omega) \rightarrow \breve{H}^{\frac{1}{2}}(\partial \Omega)$ and $\mathbb{P}_0\colon H^{\frac{1}{2}}(\partial \Omega) \rightarrow \mathcal{C}(\partial \Omega)$, $\mathbb{Q}^{\perp}_0\colon$ $H^{-\frac{1}{2}}(\partial \Omega) \rightarrow \breve{H}^{-\frac{1}{2}}(\partial \Omega)$, and $\mathbb{Q}_0\colon H^{-\frac{1}{2}}(\partial \Omega) \rightarrow \mathcal{C}(\partial \Omega)$ associated with these complementary decompositions are all bounded.}
\end{lemma}

Intuitively, the spaces $\breve{H}^{\frac{1}{2}}(\partial \Omega)$ and $\breve{H}^{-\frac{1}{2}}(\partial \Omega)$ do not contain piecewise constant functions, and this implies in particular that the Dirichlet-to-Neumann map $\text{DtN} \colon \breve{H}^{\frac{1}{2}}(\partial \Omega) \rightarrow \breve{H}^{-\frac{1}{2}}(\partial \Omega)$ is an isomorphism on these spaces. Note that these spaces can also be defined on individual spheres.

Using the projection operators defined in Definition \ref{def:projections}, it is possible to introduce new norms on the spaces $H^{\frac{1}{2}}(\partial \Omega)$ and $H^{\frac{1}{2}}(\partial \Omega)$. These new norms were first introduced in \cite{Hassan1} and will be used in a crucial way in the analysis of the current article.
\begin{definition}\label{def:NewNorm}
	We define on $H^{\frac{1}{2}}(\partial \Omega)$ a new norm $||| \cdot ||| \colon H^{\frac{1}{2}}(\partial \Omega) \rightarrow \mathbb{R}$ given by
	\begin{align*}
	\forall \lambda \in H^{\frac{1}{2}}(\partial \Omega)\colon ~ |||\lambda|||^2 := \left\Vert \mathbb{P}_0 \lambda\right\Vert^2_{L^2(\partial \Omega)}+ \left\langle \text{\rm DtN}\lambda, \lambda\right \rangle_{\partial \Omega},
	\end{align*}
	and we define on $H^{-\frac{1}{2}}(\partial \Omega)$ a new dual norm $||| \cdot |||^* \colon H^{-\frac{1}{2}}(\partial \Omega) \rightarrow \mathbb{R}$ given by
	\begin{align*}
	||| \sigma |||^* := \sup_{\psi \in H^{\frac{1}{2}}(\partial \Omega) } \frac{\left \langle \sigma, \psi \right \rangle_{\partial \Omega} }{||| \psi |||}.
	\end{align*} 
\end{definition}
As shown in \cite{Hassan1}, the norm $||| \cdot |||$ is equivalent to the $\Vert \cdot \Vert_{H^{\frac{1}{2}}(\partial \Omega)}$ norm. More precisely, there exists an $N$-independent constant $c_{\rm equiv} >1$ such that $\frac{1}{c_{\rm equiv}} ||| \lambda ||| \leq \Vert  \lambda \Vert_{H^{\frac{1}{2}}(\partial \Omega)} \leq c_{\rm equiv} ||| \lambda ||| ~~\forall \lambda \in H^{\frac{1}{2}}(\partial \Omega)$. Similarly, the new $||| \cdot |||^*$ dual norm on ${H}^{-\frac{1}{2}}(\partial \Omega)$ is equivalent to the canonical dual norm $\Vert \cdot \Vert_{H^{-\frac{1}{2}}(\partial \Omega)}$ with equivalence constant that is once again independent of $N$. Finally, it is easy to show that for all $\tilde{\lambda} \in \breve{H}^{\frac{1}{2}}(\partial \Omega)$ it holds that
\begin{align*}
||| \text{DtN} \tilde{\lambda} |||^* = ||| \tilde{\lambda} |||.
\end{align*}

In the sequel, we adopt the convention that the Hilbert space ${H}^{\frac{1}{2}}(\partial \Omega)$ is equipped with the $||| \cdot |||$ norm and that the dual space $H^{-\frac{1}{2}}(\partial \Omega)$ is equipped with the $||| \cdot |||^*$ norm. 

{ 
	\begin{definition}[Projectors on the Approximation Space]\label{def:PQ}~
		Let $\ell_{\max} \in \mathbb{N}_0$ and let the approximation space $W^{\ell_{\max}}$ be defined as in Definition \ref{def:6.7}. We define the projection operator $\mathbb{P}_{\ell_{\max}}\colon H^{\frac{1}{2}}(\partial \Omega) \rightarrow W^{\ell_{\max}}$ as the mapping with the property that for any $\psi \in H^{\frac{1}{2}}(\partial \Omega)$,  $\mathbb{P}_{\ell_{\max}}\psi $ is the unique element of $W^{\ell_{\max}}$ satisfying
		\begin{align*}
		\left(\phi_{\ell_{\max}}, \mathbb{P}_{\ell_{\max}}\psi\right)_{L^2(\partial \Omega)}=\left\langle \phi_{\ell_{\max}}, \psi\right\rangle_{\partial \Omega} \qquad \forall \phi_{\ell_{\max}} \in W^{\ell_{\max}}.
		\end{align*}
		
		Moreover, we define the projection operator $\mathbb{Q}_{\ell_{\max}}\colon H^{-\frac{1}{2}}(\partial \Omega) \rightarrow W^{\ell_{\max}}$ as the mapping with the property that for any $\sigma \in H^{-\frac{1}{2}}(\partial \Omega)$,  $\mathbb{Q}_{\ell_{\max}}\sigma $ is the unique element of $W^{\ell_{\max}}$ satisfying
		\begin{align*}
		\left(\mathbb{Q}_{\ell_{\max}}\sigma, \phi_{\ell_{\max}}\right)_{L^2(\partial \Omega)}&=\left\langle \sigma, \phi_{\ell_{\max}}\right\rangle_{\partial \Omega} \qquad \forall \phi_{\ell_{\max}} \in W^{\ell_{\max}}.
		\end{align*}
	\end{definition}
	
	\begin{remark}
		Consider the setting of Definition \ref{def:PQ}. It is possible to show that the projection operators $\mathbb{P}_{\ell_{\max}}$ and $\mathbb{Q}_{\ell_{\max}}$ are stable, i.e., for all $\psi \in H^{\frac{1}{2}}(\partial \Omega)$ and all $\sigma \in H^{-\frac{1}{2}}(\partial \Omega)$ it holds that
		\begin{align*}
		||| \mathbb{P}_{\ell_{\max}} \psi ||| \leq ||| \psi ||| \quad \text{and} \quad ||| \mathbb{Q}_{\ell_{\max}}\sigma|||^* \leq ||| \sigma |||^*.
		\end{align*}
	\end{remark}
	
}

Finally, we define the higher regularity spaces that will be used in the subsequent error estimates.
\begin{definition}\label{def:7.1}
	Let $s\geq 0$ be a real number and let $\mathcal{O}_{\bold{x}_0} \subset \mathbb{R}^3$ be an open ball of radius $r > 0$ centred at the point $\bold{x}_0 \in \mathbb{R}^3$. Then we define constructively the fractional Sobolev space ${H}^{s}(\partial\mathcal{O}_{\bold{x}})$ as the set
	\begin{align*}
	{H}^{s}(\partial\mathcal{O}_{\bold{x}_0}):= \Big\{u \colon \partial \mathcal{O}_{\bold{x}_0} \rightarrow \mathbb{R} \text{ such that } u(\bold{x})&= \sum_{{\ell}=0}^\infty \sum_{m=-\ell}^{m=+\ell} [u]^\ell_m\mathcal{Y}_{\ell}^m\left(\frac{\bold{x}-\bold{x}_0}{\vert \bold{x}-\bold{x}_0\vert}\right) \\
	\text{where all }[u]_{\ell}^m \in \mathbb{R} \text{ satisfy }& \sum_{\ell=1}^\infty\sum_{m=-\ell}^{m=+\ell}\left(\frac{l}{r}\right)^{2s}([u]_m^{\ell})^2< \infty\Big\},
	\end{align*}
	equipped with the inner product
	\begin{align}\label{eq:frac1}
	(u, v)_{{H}^{s}(\partial\mathcal{O}_{\bold{x}_0})}:= r^2[u]_0^0 \, [v]_0^0+r^2\sum_{\ell=1}^\infty\sum_{m=-\ell}^{m=+\ell}\left(\frac{\ell}{r}\right)^{2s}[u]_\ell^m [v]_\ell^m \qquad \forall u,v \in H^s(\partial \mathcal{O}_{\bold{x}_0}).
	\end{align}
	Additionally, we write $||| \cdot |||_{{H}^{s}(\partial\mathcal{O}_{\bold{x}_0})}$ to denote the norm induced by $(\cdot, \cdot)_{{H}^{s}(\partial\mathcal{O}_{\bold{x}_0})}$.
	
\end{definition}


\begin{definition}\label{def:7.2}
	Let $s\geq 0$ be a real number. Then we define the Hilbert space ${H}^{s}(\partial \Omega)$ as the set
	\begin{align*}
	{H}^{s}(\partial \Omega):= \Big\{u \colon \partial\Omega \rightarrow \mathbb{R} \text{ such that } \forall i \in \{1, \ldots, N\} \colon u\vert_{\partial \Omega_i} \in {H}^{s}(\partial \Omega_i)\Big\},
	\end{align*}
	equipped with the inner product $(u, v)_{{H}^{s}(\partial \Omega)}:= \sum_{i=1}^N \left(u, v\right)_{{H}^{s}(\partial \Omega_i)}~ \forall u, v \in H^s(\partial \Omega)$. Additionally, we write $||| \cdot |||_{{H}^{s}(\partial \Omega)}$ to denote the norm induced by $(\cdot, \cdot)_{{H}^{s}(\partial \Omega)}$.
\end{definition}

It can be verified that the norm $|||\cdot |||_{H^{\frac{1}{2}}(\partial \Omega)}$ coincides with the $||| \cdot |||$ norm defined through Definition \ref{def:NewNorm} and the norm $\Vert\cdot \Vert_{W^{\ell_{\max}}}$ on the approximation space $W^{\ell_{\max}}$.

\subsection{Previous Results}\label{sec:2c}

We state two key results from \cite{Hassan1} and \cite{Hassan2}. 

\begin{theorem}[{\cite[Theorem 2.23, Theorem 4.19]{Hassan1}}]\label{thm:well-posed}~
	Let $\sigma_f \in H^{-\frac{1}{2}}(\partial \Omega)$. There exists a unique solution $\nu \in H^{-\frac{1}{2}}(\partial \Omega)$ to the BIE \eqref{eq:3.3a} with right-hand side generated by $\sigma_f$. Additionally, let $\ell_{\max} \in \mathbb{N}$. Then there exists a unique solution $\nu_{\ell_{\max}} \in W^{\ell_{\max}}$ to the Galerkin discretisation \eqref{eq:Galerkina} with right-hand side generated by $\sigma_f$, and for all real numbers $s > -\frac{1}{2}$ we have the error bound	
	\begin{align*}
	{||| \nu-\nu_{\ell_{\max}}|||^*} \leq & C_{\rm charges}\left(\frac{\max r_j}{\ell_{\max}+1}\right)^{s+\frac{1}{2}} \left(\big|\big|\big| \mathbb{Q}_{0}^{\perp}\nu\big|\big|\big|_{{H}^{s}(\partial \Omega)} + \big|\big|\big| \mathbb{Q}_{0}^{\perp}\sigma_f\big|\big|\big|_{H^s(\partial \Omega)}\right),
	\end{align*}
	where the constant $C_{\rm charges}>1$ depends only on the dielectric constants, the radii of the spheres and the minimum inter-sphere separation distance. 
\end{theorem}

Theorem \ref{thm:well-posed} establishes errors estimates for the induced surface charge that do not explicitly depend on the number of spherical particles $N$. We can therefore conclude that given any geometry in the family of geometries $\{\Omega_{\mathcal{F}}\}_{\mathcal{F}\in \mathcal{I}}$ satisfying assumptions {\textbf{A1)-A3)}}, for a fixed number of degrees of freedom per sphere, the relative or average error in the approximate induced surface charge does not increase if the number $N_{\mathcal{F}}$ of spherical dielectric particles in the system increases. A major goal of this article is to obtain similar $N$-independent errors estimates for the electrostatic forces. We remark that a closed form expression of the pre-factor $C_{\rm charges}$ can be found in \cite[Theorem 2.22]{Hassan1}.

\begin{theorem}[{\cite[Theorem 3.12]{Hassan2}}]\label{thm:iterations}~
	Let $\ell_{\max} \in \mathbb{N}$, let $\sigma_f \in H^{-\frac{1}{2}}(\partial \Omega)$, and let $\nu_{\ell_{\max}}\in W^{\ell_{\max}}$ be the unique solution to the Galerkin discretisation \eqref{eq:Galerkina} with \mbox{right-hand} side generated by $\sigma_f$. Then for every $\epsilon > 0$ there exists a function $\nu_{\ell_{\max}}^{\rm{approx}} \in W^{\ell_{\max}}$ and a natural number $R_{\epsilon}>0$ that depends only on $\epsilon$, the dielectric constants, the radii of the spheres, and the minimum inter-sphere separation distance such that at most $R_{\epsilon}$ iterations of GMRES are required to compute $\nu_{\ell_{\max}}^{\rm{approx}} $ and such that the following error estimate holds
	\begin{align*}
	\frac{||| \nu_{\ell_{\max}}^{\rm{approx}}  - \nu_{\ell_{\max}}|||^*}{||| \mathbb{Q}_0^{\perp}\nu_{\ell_{\max}}|||^*+ \frac{4\pi}{\kappa_0}||| \mathbb{Q}_0^{\perp}\mathbb{Q}_{\ell_{\max}}\sigma_f|||^*}\,  < \epsilon.
	\end{align*}
\end{theorem}	

Theorem \ref{thm:iterations} shows that for any geometry in the family of geometries $\{\Omega_{\mathcal{F}}\}_{\mathcal{F}\in \mathcal{I}}$ satisfying assumptions {\textbf{A1)-A3)}}, it is possible to compute approximations to the discrete induced surface charge $\nu_{\ell_{\max}} \in W^{\ell_{\max}}$ up to any given relative error tolerance using a number of linear solver iterations that is independent of $N_{\mathcal{F}}$. As discussed in detail in the contributions \cite{Hassan2,lindgren2018}, the matrix vector products required by the linear solver can be done in $\mathcal{O}(N)$ using the fast multipole method (see e.g., \cite{cheng1999fast,greengard1990numerical,greengard1}). Consequently, given any geometry $\Omega_{\mathcal{F}}$ from a family of geometrical configurations $\{\Omega_{\mathcal{F}}\}_{\mathcal{F}\in \mathcal{I}}$ satisfying assumptions {\textbf{A1)-A3)}}, Theorem \ref{thm:iterations} implies that only $\mathcal{O}(N_{\mathcal{F}})$ computations are required to approximate $\nu_{\ell_{\max}}$ up to any given relative error tolerance. We remark that a closed form expression of the natural number $R_{\epsilon}$ can be found in \cite[Theorem 3.12]{Hassan2}.

\section{Definition of the Exact and Approximate Electrostatic Forces}\label{sec:3}

There are at least two approaches to defining the electrostatic forces in non-relativistic settings, one popular in the computational chemistry community and the other originating in the physics literature: Chemists tend to view the total electrostatic energy as the fundamental quantity of interest and define the electrostatic forces as functions of this energy (see, e.g., \cite{gatto2017computation,lipparini2014quantum,lu2005computation,sagui1999molecular}). Physicists on the other hand usually view Maxwell's equations for the electric and magnetic fields as the starting point of any study of electromagnetic phenomena (see, e.g., \cite{feynman,griffiths_2017,panofsky}). In this formalism, the electromagnetic force is given by the so-called Lorentz force law (see, e.g., \cite[Chapter 27]{feynman} or \cite[Chapter 10]{panofsky}) which defines the force in terms of the \emph{electric and magnetic fields}. Naturally, in the absence of electrodynamic effects and magnetic fields, the electrostatic force is defined purely in terms of the electric field.

Although both definitions seemingly arise from different physical considerations, it is possible to show that they are in fact equivalent (see Appendix \ref{sec:appendix}). For the purpose of numerical analysis however, it is advantageous to use the electric-field based formalism to define the forces, and this is the convention we adopt. In the sequel, we assume the setting of Section \ref{sec:2} and we assume that each spherical dielectric particle is a uniform rigid body so that we need consider only the \emph{net} force acting on each spherical particle.





\begin{definition}[Electric Potential]\label{def:pot}~
	
	Let $\sigma \in H^{-\frac{1}{2}}(\partial \Omega)$ be a charge distribution supported on the boundary $\partial \Omega$ of the collection of open balls $\{\Omega_i\}_{i=1}^N$ and let $\mathcal{S} \colon H^{-\frac{1}{2}}(\partial \Omega) \rightarrow H^1_{\rm loc}(\R^3)$ be the single layer potential . Then we define the function $\phi \in H^1_{\rm loc}(\mathbb{R}^3)$ as
	\begin{align*}
	\phi:=\mathcal{S} \sigma,
	\end{align*}
	and we say that $\phi$ is the electric potential produced by the charge distribution $\sigma$.
\end{definition}

Note that the potential $\phi$ is typically the quantity of interest if one formulates the dielectric spheres electrostatic interaction problem, i.e., the BIE \eqref{eq:3.3a} as a PDE-based transmission problem (see, e.g., \cite{Hassan1,lindgren2018}). Furthermore, recalling the definition of the single layer boundary operator $\mathcal{V}$ from Section \ref{sec:2a}, we observe that the interior and exterior Dirichlet traces $\gamma^{\pm} \mathcal{S}\sigma= \mathcal{V}\sigma$ are well-defined.

\begin{definition}[Electric Field]\label{def:field}~
	
	Let $\sigma \in H^{-\frac{1}{2}}(\partial \Omega)$ be a charge distribution supported on the boundary $\partial \Omega$ of the collection of open balls $\{\Omega_i\}_{i=1}^N$, and let $\phi \in H_{\rm loc}^1(\mathbb{R}^3)$ be the electric potential produced by the charge distribution $\sigma$. Then we define the vector field $\boldsymbol{E} \in L^2_{\rm loc}(\mathbb{R}^3; \mathbb{R}^3)$ as
	\begin{align*}
	\boldsymbol{E}:= -\nabla \phi,
	\end{align*}
	and we say that $\boldsymbol{E}$ is the electric field produced by the charge distribution $\sigma$. Here, $\nabla$ denotes the usual gradient in cartesian coordinates.
\end{definition}

\begin{remark}
	Consider the setting of Definition \ref{def:field}. Since the electric potential $\phi$ is harmonic on the complement of the boundary $\partial \Omega$, it is in fact smooth on $\mathbb{R}^3 \setminus \partial \Omega$. Consequently, the electric field $\boldsymbol{E}$ is point-wise infinitely differentiable at any $\bold{x} \in \mathbb{R}^3 \setminus \partial \Omega$.
\end{remark}

Let now $i \in \{1, \ldots, N\}$. It will be important to consider also the electric field generated by a charge distribution supported only on the collection of spheres $\{\partial \Omega_j\}_{\substack{j=1, \\ j \neq i}}^N$, i.e., \emph{excluding the sphere $\partial \Omega_i$}. To this end, we first introduce some simplifying notation.\\

\noindent {\textbf{Notation}: } Let $i \in \{1, \ldots, N\}$. We define the set $\partial \omega_i \subset \partial \Omega$  as $\partial \omega_i:= \partial \Omega \setminus \partial \Omega_i$. In other words $\partial \omega_i$ is the boundary of the collection of open balls $\{\Omega_j\}_{\substack{j=1, \\ j \neq i}}^N$, i.e., \emph{excluding the open ball $\Omega_i$}.

\begin{definition}[Excluded Electric Potentials and Fields]\label{def:field2}~
	
	Let $\sigma \in H^{-\frac{1}{2}}(\partial \Omega)$ be a charge distribution supported on the boundary $\partial \Omega$ of the collection of open balls $\{\Omega_i\}_{i=1}^N$. Then for each $i \in \{1, \ldots, N\}$ 
	\begin{itemize}
		\item We define ${\sigma}_{i,{\rm exc}} \in H^{-\frac{1}{2}}(\partial \Omega)$ as 
		\begin{equation*}
		{\sigma}_{i,{\rm exc}}:= \begin{cases}
		\sigma \qquad &\text{on } \partial \omega_i,\\
		0\qquad &\text{on } \partial \Omega_i,
		\end{cases}
		\end{equation*}
		and we say that ${\sigma}_{i, {\rm exc}}$ is the $i$ excluded charge distribution;
		
		\item We define the function $\phi_{i,{\rm exc}} \in H^1_{\rm loc}(\mathbb{R}^3)$ as 
		\begin{equation*}
		\phi_{i,{\rm exc}} := \mathcal{S}\sigma_{i,{\rm exc}},
		\end{equation*}
		and we say that $\phi_{i,{\rm exc}}$ is the $i$ excluded electric potential generated by $\sigma$;
		
		\item We define the vector field $\boldsymbol{E}_i \in L_{\rm loc}^2(\mathbb{R}^3; \mathbb{R}^3)$ as
		\begin{equation}
		\boldsymbol{E}_{i}:= -\nabla \phi_{i,{\rm exc}},
		\end{equation}
		and we say that $\boldsymbol{E}_i$ is the $i$ excluded electric field generated by $\sigma$.
	\end{itemize} 
\end{definition}

Two remarks are now in order.

\begin{remark}\label{rem:Field_interpretation}
	Consider the setting of Definition \ref{def:field2}. The vector field $\boldsymbol{E}_i$, i.e., the $i$ excluded electric field generated by $\sigma$ has a physical interpretation. This is precisely the part of the total electric field that interacts with (i.e., exerts a net electrostatic force on) the charge distribution on the spherical dielectric particle represented by $\Omega_i$. 
\end{remark}

\begin{remark}\label{rem:Field1}
	Consider the setting of Definition \ref{def:field2} and let $i \in \{1, \ldots, N\}$. We observe that the $i$ excluded potential $\phi_{i,\rm{exc}}$ is harmonic and therefore smooth on the complement of  $\partial \omega_i$. This implies that the $\phi_{i,\rm{exc}}$ is smooth on the boundary $\partial \Omega_i$, i.e., on the surface of the $i^{\rm th}$ spherical dielectric particle. Consequently, the $i$ excluded electric field $\boldsymbol{E}_i$ is also smooth on the boundary $\partial \Omega_i$. 
\end{remark}

We are now ready to define the electrostatic force acting on each spherical dielectric particle. As mentioned previously the force is defined using the well-known Lorentz force law.

\begin{definition}[Definition of the Forces]\label{def:Force2}~
	
	Let $\sigma_f \in H^{-\frac{1}{2}}(\partial \Omega)$ be a given free charge, let $\nu \in H^{-\frac{1}{2}}(\partial \Omega)$ denote the unique solution to the BIE \eqref{eq:3.3a} with right-hand side generated by $\sigma_f$, and for each $i \in \{1, \ldots, N\}$ let $\boldsymbol{E}_{i} \in L_{\rm loc}^2(\mathbb{R}^3; \mathbb{R}^3)$ denote the $i$ excluded electric field generated by $\nu$ as defined through Definition \ref{def:field2}. Then for each $i=1, \ldots, N$ we define the net force acting on the dielectric particle represented by $\Omega_i$ as the vector $\boldsymbol{F}_i \in \mathbb{R}^3$ given by
	\begin{align*}
	\boldsymbol{F}_i:= \kappa_0\int_{\partial \Omega_i} \nu(\bold{x}) \boldsymbol{E}_i(\bold{x})\, d \bold{x}.
	\end{align*}
\end{definition}

\begin{remark}
	Consider Definition \ref{def:Force2} of the electrostatic forces. We remark that since $\nu \in H^{-\frac{1}{2}}(\partial \Omega)$, the integral should be understood as a $\langle \cdot, \cdot \rangle_{H^{-\frac{1}{2}}(\partial \Omega_i) \times H^{\frac{1}{2}}(\partial \Omega_i)}$ duality pairing. In view of Remark \ref{rem:Field1}, the $i$ excluded electric field $\boldsymbol{E}_i$ is smooth on $\partial \Omega_i$ so the duality pairing is well-defined.
\end{remark}


Clearly, Definition \ref{def:Force2} of the electrostatic forces cannot be used for practical computations since it relies on knowledge of exact quantities of interest. Therefore, it is necessary to define \emph{approximate} electrostatic forces using approximate $i$ excluded electric potentials and electric fields for all $i=1, \ldots, N$. 

\begin{definition}[Definition of the Approximate Forces]\label{def:Force2_approx}~
	Let $\sigma_f \in H^{-\frac{1}{2}}(\partial \Omega)$ be a given free charge, let ${\ell_{\max}} \in \mathbb{N}$, let $\nu_{\ell_{\max}} \in W^{\ell_{\max}}$ be the unique solution to the Galerkin discretisation \eqref{eq:Galerkina} with right-hand side generated by $\sigma_f$, and for each $i \in \{1, \ldots, N\}$ let $\boldsymbol{E}^{\ell_{\max}}_{i} \in L_{\rm loc}^2(\mathbb{R}^3; \mathbb{R}^3)$ denote the $i$ excluded electric field generated by $\nu_{\ell_{\max}}$ as defined through Definition \ref{def:field2}. Then for each $i=1, \ldots, N$ we define the approximate net force acting on the dielectric particle represented by $\partial \Omega_i$ as the vector $\boldsymbol{F}^{\ell_{\max}}_i \in \mathbb{R}^3$ given by
	\begin{align}\label{eq:def_force}
	\boldsymbol{F}_i^{\ell_{\max}}:= \kappa_0\int_{\partial \Omega_i} \nu_{\ell_{\max}}(\bold{x}) \boldsymbol{E}^{\ell_{\max}}_i(\bold{x})\, d \bold{x}.
	\end{align}
\end{definition}


\section{Error Analysis for the Electrostatic Forces}\label{sec:4}

Throughout this section we will assume the setting of Section \ref{sec:2}. In order to present a clear and concise exposition, we organise the remainder of this section as follows. In Section \ref{sec:4a}, we state our main results on the approximability and rate of convergence of the approximate electrostatic forces, and we discuss the hypothesis and conclusions of these theorems. In Section \ref{sec:4b}, we state and prove intermediary lemmas which we require for our analysis. These lemmas are then used to prove the main results of Section \ref{sec:4a}.

\subsection{Main Result and Discussion}\label{sec:4a}

\noindent \textbf{Notation:} Given a vector $\mathbb{X} \in \R^3$, we will write $(\mathbb{X})_{\alpha}, ~\alpha =1,2,3$ to denote the components of $\mathbb{X}$. In the same spirit, we will write $\partial_{\alpha} (\cdot), ~\alpha =1,2,3$ for the $\alpha^{\rm th}$ component of the gradient of some scalar field in the cartesian coordinate system.

We begin with a standard result on the approximability of the electrostatic forces. 

\begin{theorem}[Approximability of the Electrostatic Forces]\label{lem:hassan4}~
	
	\noindent Let $s >-\frac{1}{2}$, let $\sigma_f \in H^{s}(\partial \Omega)$ be a given free charge, let $\ell_{\max} \in \mathbb{N}$, let $\nu \in H^{s}(\partial \Omega)$ and $\nu_{\ell_{\max}} \in W^{\ell_{\max}}$ denote the unique solutions to the BIE \eqref{eq:3.3a} and Galerkin discretisation \eqref{eq:Galerkina} with right-hand sides generated by $\sigma_f$, and for each $i \in \{1, \ldots, N\}$ let $\boldsymbol{F}_i, \boldsymbol{F}^{\ell_{\max}}_{i} \in  \mathbb{R}^3$ denote the exact and approximate force acting on the dielectric particle represented by $ \Omega_i$ as defined through Definitions \ref{def:Force2} and \ref{def:Force2_approx} respectively. Then it holds that
	\begin{align*}
	\lim_{\ell_{\max} \to \infty}\sum_{i=1}^N\sum_{\alpha=1}^3 \Big\vert (\boldsymbol{F}_i)_{\alpha} - \big(\boldsymbol{F}_i^{\ell_{\max}}\big)_{\alpha}\Big \vert =0.
	\end{align*}	
\end{theorem}

\begin{remark}
	Consider Theorem \ref{lem:hassan4}. Notice that we require that the free charge $\sigma_f \in H^{s}(\Omega)$ for some $s > -\frac{1}{2}$. This regularity requirement agrees with the convergence rates for the induced surface charge given by Theorem \ref{thm:well-posed} and therefore cannot be improved. Unfortunately, as we will show in Section \ref{sec:4b}, the techniques used to prove this theorem lead to convergence rates with a pre-factor that depends on $N$. 
\end{remark}

The next theorem establishes $N$-independent convergence rates for the electrostatic forces under the assumption of increased regularity for the free charge $\sigma_f$.

\begin{theorem}[$N$-independent Convergence Rates for the Electrostatic Forces]~\label{thm:Force}~
	\noindent Let $s \geq \frac{1}{2}$, let $\sigma_f \in H^{s}(\partial \Omega)$ be a given free charge, let $\ell_{\max} \in \mathbb{N}$,  let $\nu \in H^{s}(\partial \Omega)$ and $\nu_{\ell_{\max}} \in W^{\ell_{\max}}$ be the unique solutions to the BIE \eqref{eq:3.3a} and Galerkin discretisation \eqref{eq:Galerkina} with right-hand sides generated by $\sigma_f$, and for each $i=1, \ldots, N$ let $\boldsymbol{F}_i$ and $\boldsymbol{F}_i^{\ell_{\max}}$ denote the net force and approximate net force as defined through Definitions \ref{def:Force2} and \ref{def:Force2_approx} respectively. Then there exists a constant $C_{\rm Force}>0$ that depends on $s$, the dielectric constants, the radii of the open balls, and the minimum inter-sphere separation distance but is independent of the number $N$ of dielectric particles such that for $\ell_{\max}$ sufficiently large it holds that
	\begin{multline}\label{eq:rates_force}
	\sum_{i=1}^N \sum_{\alpha=1}^3\Big\vert (\boldsymbol{F}_i)_{\alpha} - \big(\boldsymbol{F}_i^{\ell_{\max}}\big)_{\alpha}\Big \vert \leq C_{\rm force} \Big(\frac{\max r_j}{\ell_{\max}+1}\Big)^{-\frac{1}{2}+s}\left(\big|\big|\big|\nu\big|\big|\big|_{H^s(\partial \Omega)}+\big|\big|\big| \mathbb{Q}_{0}^{\perp}\sigma_f\big|\big|\big|_{H^s(\partial \Omega)}\right)^2.
	\end{multline}
\end{theorem}
\begin{remark}
	Consider the setting of Theorem \ref{thm:Force}. The dependence of the pre-factor $C_{\rm Force}$ on the regularity index $s \geq \frac{1}{2}$ is an artefact which arises due to our attempt to write the final convergence rates \eqref{eq:rates_force} in a concise and aesthetically appealing form. The proof of Theorem \ref{thm:Force} will show that the dependence of the pre-factor $C_{\rm Force}$ on $s$ can be removed at the cost of obtaining a more tedious final expression.
\end{remark}

Theorem \ref{thm:Force} has the following important (but unsurprising) corollary.

\begin{corollary}[Exponential Convergence of the Electrostatic Forces]\label{thm:exp}~
	\noindent Let $C_{\rm force}$ denote the convergence rate pre-factor in Theorem \ref{thm:Force}, let $\sigma_f \in C^\infty(\partial \Omega)$ be such that the harmonic extension of $\sigma_f$ inside $\Omega^-$ is analytic on $\overline{\Omega^-}$, let $\ell_{\max} \in \mathbb{N}$, let $\nu \in {H}^{-\frac{1}{2}}(\partial \Omega)$ and $\nu_{\ell_{\max}} \in W^{\ell_{\max}}$ be the unique solutions to the BIE \eqref{eq:3.3a} and Galerkin discretisation \eqref{eq:Galerkina} with right-hand sides generated by $\sigma_f$, and for each $i \in \{1, \ldots, N\}$ let $\boldsymbol{F}_i$ and $\boldsymbol{F}_i^{\ell_{\max}}$ denote the exact and approximate electrostatic forces acting on the particle represented by $\Omega_i$ as given by Definitions \ref{def:Force2} and \ref{def:Force2_approx} respectively. For $\ell_{\max}$ sufficiently large, if the harmonic extension of $\nu$ inside $\Omega^-$ is analytic on $\overline{\Omega^-}$ then there exist constants $C^{\alpha}_{\nu, \sigma_f}, C^{\beta}_{\nu, \sigma_f}  > 0$ depending on geometric parameters, the exact solution $\nu$, and the free charge $\sigma_f$ such that
	\begin{align*}
	\frac{1}{{N}}\sum_{i=1}^N\sum_{\alpha=1}^3 \Big\vert (\boldsymbol{F}_i)_{\alpha} - \big(\boldsymbol{F}_i^{\ell_{\max}}\big)_{\alpha}\Big \vert  &\leq C_{\rm force}C^{\alpha}_{\nu, \sigma_f}\exp\left(-C^{\beta}_{\nu, \sigma_f}\sqrt{\frac{\ell_{\max}+1}{\max r_j}}\right).
	\end{align*}
\end{corollary}
\begin{proof}
	The proof of Corollary \ref{thm:exp} uses the convergence rates obtained from Theorem \ref{thm:Force} and is a standard exercise in the analysis of spectral Galerkin methods (see, e.g., the arguments in \cite{houston2000stabilized,houston2001hp}). For more details, we refer the interested reader to the proof of \cite[Theorem 2.24]{Hassan1} which establishes the exponential convergence of the approximate induced surface charge and can be copied nearly word for word to prove Corollary \ref{thm:exp}.
\end{proof}

Next, we would like to discuss in more detail the hypothesis and conclusions of Theorem \ref{thm:Force}. We frame this discussion in the form of two remarks.

{ \begin{remark}[\textit{Scaling of the Error Estimates in Theorem} \ref{thm:Force}]\label{rem:thm1}~
		
		\noindent	As mentioned in the introduction, the goal of this work is to show that, under suitable geometric assumptions, the Galerkin method proposed through the BIE \eqref{eq:3.3a} and the discretisation \eqref{eq:Galerkina} can be used to obtain the electrostatic forces with \emph{linear scaling accuracy}. A necessary condition to achieve this is to show that the approximate forces we compute are $N$-error stable, i.e., for a fixed number of degrees of freedom per sphere, the relative or average error in the approximate approximate forces for different geometrical configurations does not increase with the number $N$ of dielectric particles. In the earlier contribution \cite{Hassan1}, we showed precisely this result for the approximate induced surface charges induced on any family of geometries $\{\Omega_{\mathcal{F}}\}_{\mathcal{F} \in \mathcal{I}}$ satisfying assumptions \textbf{A1)-A3)}.
		
		Consider Theorem \ref{thm:Force}. Since the constant $C_{\rm force}>0$ does not explicitly depend on $N$, we can deduce two conclusions from the convergence rate \eqref{eq:rates_force}:
		
		\begin{itemize}
			\item[C1)] It holds that
			\begin{align*}
			\frac{\sum_{i=1}^N \sum_{\alpha=1}^3\Big\vert (\boldsymbol{F}_i)_{\alpha} - \big(\boldsymbol{F}_i^{\ell_{\max}}\big)_{\alpha}\Big \vert}{ \left(\Vert \nu\Vert_{H^s(\partial \Omega)}+\Vert \mathbb{Q}_{0}^{\perp}\sigma_f\Vert_{H^s(\partial \Omega)}\right)^2} \leq \Big(\frac{\max r_j}{\ell_{\max}+1}\Big)^{-\frac{1}{2}+s} C_{\rm force},
			\end{align*}
			with right-hand side that is independent of $N$. Consequently, for any family of geometries $\{\Omega_{\mathcal{F}}\}_{\mathcal{F} \in \mathcal{I}}$ satisfying assumptions \textbf{A1)-A3)}, the approximate forces-- relative to the sum squared of the exact induced surface charge $\nu$ and free charge $\sigma_f$-- are indeed \underline{independent of $N_{\mathcal{F}}$}. Here, $N_{\mathcal{F}}$ denotes the number of particles in an arbitrary geometry $\Omega_{\mathcal{F}}$ and takes the role of $N$ above. Unfortunately, we have been unable to obtain such a result for the error in the approximate forces relative to the exact force.

			\item[C2)] Let $i \in \{1, \ldots, N\}$ and let $\nu_i:= \nu \vert_{\partial \Omega_i}$ and $\sigma_{f, i}:= \sigma_f \vert_{\partial \Omega_i}$. If the induced surface charge $\nu$ and free charge $\sigma_f$ are both of order 1 on each sphere, i.e., if $|||\nu_i|||_{H^s(\partial \Omega_i)},~ |||\sigma_{f, i}|||_{H^s(\partial \Omega_i)}= \mathcal{O}(1)$ then it holds that
			\begin{align*}
			\left(\big|\big|\big|\nu\big|\big|\big|_{H^s(\partial \Omega)}+\big|\big|\big| \mathbb{Q}_{0}^{\perp}\sigma_f\big|\big|\big|_{H^s(\partial \Omega)}\right)^2= \mathcal{O}(N),
			\end{align*}
			and it therefore holds that
			\begin{align*}
			\frac{1}{{N}}\sum_{i=1}^N\sum_{\alpha=1}^3 \Big\vert (\boldsymbol{F}_i)_{\alpha} - \big(\boldsymbol{F}_i^{\ell_{\max}}\big)_{\alpha}\Big \vert  = \mathcal{O}(1).
			\end{align*}
			
			In other words, for any family of geometries $\{\Omega_{\mathcal{F}}\}_{\mathcal{F} \in \mathcal{I}}$ satisfying assumptions \textbf{A1)-A3)}, if the discretisation parameter $\ell_{\max}$ is fixed and the induced surface charge $\nu$ and free charge $\sigma_f$ are both $\mathcal{O}(1)$ on each sphere, then the average error in the approximate forces does not increase for increasing $N_{\mathcal{F}}$.
		\end{itemize}
	\end{remark}
}
\begin{remark}[\textit{Assumptions of Theorem} \ref{thm:Force}]\label{rem:thm2}~
	
	\noindent Consider once again Theorem \ref{thm:Force}. Notice that we require that the free charge $\sigma_f \in H^s(\partial \Omega)$ for $s > \frac{1}{2}$. This is in contrast to Theorem \ref{lem:hassan4}, which establishes approximability of the forces even if $-\frac{1}{2} < s < \frac{1}{2} $. Although suboptimal from a mathematical perspective, this additional regularity assumption does not preclude us from using these error estimates in most practical situations. This is due to the fact that in many physical applications, the open balls $\{\Omega_i\}_{i=1}^N$ represent \emph{homogenous} dielectric particles, and for homogenous particles, physical arguments imply that the free charge $\sigma_f$ must be distributed uniformly on each sphere $\partial \Omega_i$. In other words, for a variety of physical applications (see, e.g., \cite{lindgren2019theoretical,lindgren2018,Titan,lindgren2018dynamic}), we have $\sigma_f \in C^{\infty}(\partial \Omega)$ which in turn means that the convergence rates \eqref{eq:rates_force} for the electrostatic forces are valid for all $s > \frac{1}{2}$. 
\end{remark}

\subsection{Auxiliary Lemmas and Proofs of the Main Results}\label{sec:4b}

{To aid the analysis of this section, we first introduce some additional notation. Essentially, we wish to introduce local versions of the projection operators (Definition \ref{def:PQ}), norms (Definition \ref{def:NewNorm}), and the trace operator on each sphere. \\
	
	\noindent \textbf{Notation:} 
	
	\begin{itemize}
		\item Let $i \in \{1, \ldots, N\}$ and $\ell \in \mathbb{N}_0$. We define the projection operator $\mathbb{P}_{i, \ell} \colon$ $H^{\frac{1}{2}}(\partial \Omega_i) \rightarrow W^{\ell}(\partial \Omega_i)$ as the mapping with the property that for any $\psi \in H^{\frac{1}{2}}(\partial \Omega)$,  $\mathbb{P}_{i, \ell} \psi $ is the unique element of $W^{\ell}(\partial \Omega_i)$ satisfying
		\begin{align*}
		\left(\phi_{\ell_{\max}}, \mathbb{P}_{i, \ell} \psi\right)_{L^2(\partial \Omega)}=\left\langle \phi_{\ell_{\max}}, \psi\right\rangle_{\partial \Omega_i} \qquad \forall \phi_{\ell_{\max}} \in W^{\ell}(\partial \Omega_i),
		\end{align*}
		
		Similarly, we define the projection operator  $\mathbb{Q}_{i, \ell} \colon H^{-\frac{1}{2}}(\partial \Omega_i) \rightarrow W^{\ell}(\partial \Omega_i)$ as the mapping with the property that for any $\sigma \in H^{-\frac{1}{2}}(\partial \Omega_i)$,  $\mathbb{Q}_{i, \ell} \sigma $ is the unique element of $W^{\ell_{\max}}(\partial \Omega)$ satisfying
		\begin{align*}
		\left(\mathbb{Q}_{i, \ell} \sigma, \phi_{\ell_{\max}}\right)_{L^2(\partial \Omega)}&=\left\langle \sigma, \phi_{\ell_{\max}}\right\rangle_{\partial \Omega_i} \qquad \forall \phi_{\ell_{\max}} \in W^{\ell}(\partial \Omega_i).
		\end{align*}

		\item Let $i \in \{1, \ldots, N\}$ and $\ell \in \mathbb{N}_0$. We define the projection operators $\mathbb{P}_{i, \ell}^{\perp} \colon$ $H^{\frac{1}{2}}(\partial \Omega_i)\rightarrow \big(W^{\ell_{\max}}(\partial \Omega_i)\big)^{\perp}$ and $\mathbb{Q}_{i, \ell}^{\perp} \colon H^{-\frac{1}{2}}(\partial \Omega_i)\rightarrow \big(W^{\ell_{\max}}(\partial \Omega_i)\big)^{\perp}$ as \newline $\mathbb{P}_{i, \ell}^{\perp}:= \mathbb{I} - \mathbb{P}_{i, \ell}$ and $\mathbb{Q}_{i, \ell}^{\perp} := \mathbb{I} - \mathbb{Q}_{i, \ell}$ where $\mathbb{I}$ denotes the identity operator on the relevant trace space. 
		
		\item Furthermore, for all $i \in \{1, \ldots, N\}$ and $\lambda_i \in H^{\frac{1}{2}}(\partial \Omega_i)$ we define
		\begin{align*}
		|||\lambda_i|||^2_i:=& \Vert \mathbb{P}_{i, 0}\lambda_i\Vert^2_{L^2(\partial \Omega_i)} + \langle \text{DtN}\lambda_i, \lambda_i\rangle_{\partial \Omega_i}.
		\end{align*}	
		
		\item In addition, for all $i \in \{1, \ldots, N\}$ and $\sigma_i \in \breve{H}^{-\frac{1}{2}}(\partial \Omega_i)$ we define
		\begin{align*}
		|||\sigma_i|||_i^*:=& |||\text{DtN}_i^{-1}\sigma_i|||_i,
		\end{align*}
		where the mapping $\text{DtN}_i^{-1} \colon \breve{H}^{-\frac{1}{2}}(\partial \Omega_i) \rightarrow \breve{H}^{\frac{1}{2}}(\partial \Omega_i)$ is the inverse of the Dirichlet-to-Neumann map on $\partial \Omega_i$.
		
		\item Moreover, for all $i \in \{1, \ldots, N\}$ we denote by $\gamma^-_i \colon H^1(\Omega_i) \rightarrow H^{\frac{1}{2}}(\partial \Omega_i)$ the interior Dirichlet trace operator on the open ball $\Omega_i$.
	\end{itemize}		
}

\begin{lemma}[Estimates for the Excluded Electric Field]\label{lem:hassan}~
	\noindent 	Let $\sigma \in H^{-\frac{1}{2}}(\partial \Omega)$ be a given charge distribution and for each $i \in \{1, \ldots, N\}$ let $\phi_{i,{\rm exc}} \in H_{\rm loc}^1(\mathbb{R}^3)$ and $\boldsymbol{E}_{i} \in L^2_{\rm loc}(\mathbb{R}^3; \mathbb{R}^3)$ denote, respectively, the $i$ excluded electric potential and electric field generated by $\sigma$ as defined through Definition \ref{def:field2}. Then for all $i \in \{1, \ldots, N\}$ there exists a constant $\widetilde{C_{r_i}}> 0$ that depends only on the radius $r_i$ of the open ball $\Omega_i$ such that for each $\alpha=1,2,3$ it holds that
	\begin{subequations}
		\begin{align}\label{eq:New_Ben_1}
		|||\gamma^-_i \big(\boldsymbol{E}_{i}\big)_{\alpha}   |||_i^2 &\leq \widetilde{C_{r_i}} \big|\big|\big|\mathbb{P}^{\perp}_{0, i}\gamma_i^- \phi_{i,{\rm exc}}\big|\big|\big|^2_{H^{\frac{3}{2}}(\partial \Omega_i)}\\ \label{eq:New_Ben_2}
		|||\mathbb{P}_{\ell_{\max}, i}\gamma^-_i \big(\boldsymbol{E}_{i}\big)_{\alpha}   |||_i^2 &\leq \widetilde{C_{r_i}} \big|\big|\big|\mathbb{P}_{\ell_{\max}+1, i}\mathbb{P}_0^{\perp}\gamma_i^- \phi_{i,{\rm exc}}\big|\big|\big|^2_{H^{\frac{3}{2}}(\partial \Omega_i)}.
		\end{align}
	\end{subequations}
\end{lemma}

\begin{remark}
	Consider Lemma \ref{lem:hassan}. We recall from Remark \ref{rem:Field1} that the $i$ excluded electric potential $\phi_{i, {\rm exc}}$ and electric field $\boldsymbol{E}_{i}$ are both smooth on $\overline{\Omega_i}$ for each $i \in \{1, \ldots,N\}$. Consequently, the norms appearing in the estimate are well-defined. Furthermore, although the conclusion of Lemma \ref{lem:hassan} might seem obvious, the key novelty of this result is that the constant $\widetilde{C_{r_i}}$ appearing in the bound depends only on the radius $r_i$ of the sphere $\Omega_i$ and is, in particular, independent of the number $N$ of dielectric spherical particles.
\end{remark}

\begin{proof}[Lemma \ref{lem:hassan}:]
	We will first prove the estimate \ref{eq:New_Ben_1}. Let $\{i \in 1\, \ldots, N\}$ be fixed. As emphasised previously, Remark \ref{rem:Field1} implies that $\phi_{i, {\rm exc}} \in C^{\infty}(\overline{\Omega_i})$. In view of Definition \ref{def:7.2} therefore, there exist coefficients $[\phi_i]_{\ell}^m$, $\ell \in \mathbb{N}_0$, $-\ell \leq m \leq \ell$ satisfying suitable decay properties such that for all $\bold{x} \in \partial \Omega_i$ it holds that
	\begin{align*}
	\phi_{i, {\rm exc}}(\bold{x})= \sum_{{\ell}=0}^{\infty}\sum_{m=-\ell}^{m=\ell} [\phi_i]_{\ell}^m \mathcal{Y}_{\ell}^m \left(\frac{\bold{x}-\bold{x}_i}{\vert \bold{x}-\bold{x}_i \vert}\right).
	\end{align*}
	
	Since the electric potential $\phi_{i, {\rm exc}}$ is harmonic on $\Omega_i$, standard results (see, e.g., \cite[Chapter 2(H)]{Folland}) yield that we have the following representation of $\phi_{i, {\rm exc}}$ in $\overline{\Omega_i}$:
	\begin{align}\label{eq:Hassan1}
	\phi_{i, {\rm exc}}= \sum_{{\ell}=0}^{\infty}\sum_{m=-\ell}^{m=\ell} [\phi_i]_{\ell}^m \frac{\vert\bold{x}-\bold{x}_i\vert^{\ell} }{r_i^{\ell}}\mathcal{Y}_{\ell}^m \left(\frac{\bold{x}-\bold{x}_i}{\vert \bold{x}-\bold{x}_i \vert}\right) \quad \forall \bold{x} \in \overline{\Omega_i}.
	\end{align}
	
	We recall from Definition \ref{def:field2} of $\boldsymbol{E}_i$ that $\gamma_i^- \big(\boldsymbol{E}_i\big)_{\alpha} = \gamma_i^- \partial_{\alpha} \phi_{i,{\rm exc}}$ for each $\alpha=1,2,3$. Since $\phi_{i,{\rm exc}}$ is smooth in a neighbourhood of the sphere $\Omega_i$ (see Remark \ref{rem:Field1}), we can use Equation \eqref{eq:Hassan1} to obtain an expression for the derivative in $\overline{\Omega_i}$ and then simply take the restriction of this derivative on the boundary $\partial \Omega_i$. To this end, we first observe that for all $\bold{x} \in \overline{\Omega_i}$ it holds that
	\begin{align}\label{eq:Hassan3}
	\big(\partial_{\alpha}\phi_{i,{\rm exc}}\big)(\bold{x}) = \sum_{{\ell}=0}^{\infty}\sum_{m=-\ell}^{m=\ell} \frac{1}{r_i^{\ell}}[\phi_i]_{\ell}^m\, \partial_{\alpha}\left(\vert\bold{x}-\bold{x}_i\vert^{\ell} \mathcal{Y}_{\ell}^m \left(\frac{\bold{x}-\bold{x}_i}{\vert \bold{x}-\bold{x}_i \vert}\right)\right).
	\end{align}
	
	Clearly, the next step would be to bound the quantity $||| \gamma_i^-\big(\partial_{\alpha}\phi_{i,{\rm exc}}\big)|||_i$ using the expression given by Equation \eqref{eq:Hassan3}. However, in order to evaluate the $||| \cdot |||_i$ norm, we require the projection of the trace onto the space of constant functions $\mathcal{C}(\partial \Omega_i)$. Therefore, our next step is to simplify the series expansion above.
	
	Observe that the $\ell=0$ term in the series expansion \eqref{eq:Hassan3} is zero. Furthermore, the $\ell=1$ terms are easy to simplify. Indeed, using the definition of the spherical harmonics in cartesian coordinates, we have
	\begin{align*}
	\partial_{\alpha} \vert\bold{x}-\bold{x}_i\vert \mathcal{Y}_{1}^m \left(\frac{\bold{x}-\bold{x}_i}{\vert \bold{x} - \bold{x}_i\vert}\right)=\begin{cases}\sqrt{\frac{3}{4\pi}} \qquad &\text{ if } (m, \alpha) \in \big\{(-1, 2), (0, 3), (1,1)\big\},\\
	0 \qquad &\text{ otherwise}.
	\end{cases}
	\end{align*}
	
	Therefore, with the introduction of an appropriate binary-valued map $k(m, \alpha) \in \{0, 1\}$ we have for all $\bold{x} \in \overline{\Omega_i}$ that
	\begin{align*}
	\big(\partial_{\alpha}\phi_{i,{\rm exc}}\big)(\bold{x}) = \sum_{{\ell}=2}^{\infty}\sum_{m=-\ell}^{m=\ell} \frac{1}{r_i^{\ell}}[\phi_i]_{\ell}^m\, \partial_{\alpha}\left(\vert\bold{x}-\bold{x}_i\vert^{\ell} \mathcal{Y}_{\ell}^m \left(\frac{\bold{x}-\bold{x}_i}{\vert \bold{x}-\bold{x}_i \vert}\right)\right)+\underbrace{\sum_{m=-1}^{m=1}\sqrt{\frac{3}{4\pi r^2_i}}[\phi_i]_{1}^m\, k(m, \alpha)}_{:= C_{\rm const}}.
	\end{align*}
	
	We have thus decomposed $\partial_{\alpha}\phi_{i,{\rm exc}}$ on the ball $\overline{\Omega_i}$ as the sum of two terms, one of which is a constant $C_{\rm const}$. We claim that in fact $\mathbb{P}_{0, i}\left(\gamma^-_i\left(\partial_{\alpha}\phi_{i,{\rm exc}}\right)\right)= C_{\rm const}$. To justify the claim, it suffices to consider the function $\psi^{\alpha}_i:=\partial_{\alpha}\phi_{i,{\rm exc}} -C_{\rm const}$ and show that $\psi^{\alpha}_i \in \breve{H}^{\frac{1}{2}}(\partial \Omega_i)$. We observe that
	\begin{align*}
	\psi^{\alpha}_i(\bold{x})=\sum_{{\ell}=2}^{\infty}\sum_{m=-\ell}^{m=\ell} \frac{1}{r_i^{\ell}}[\phi_i]_{\ell}^m\, \partial_{\alpha}\left(\vert\bold{x}-\bold{x}_i\vert^{\ell} \mathcal{Y}_{\ell}^m \left(\frac{\bold{x}-\bold{x}_i}{\vert \bold{x}-\bold{x}_i \vert}\right)\right)~\forall \bold{x} \in \overline{\Omega_i}.
	\end{align*}
	
	It is well known (see, e.g., \cite[Chapter 5]{axler2013harmonic}, \cite[Chapter 8]{McLean}) that the function $
	\vert\bold{x}-\bold{x}_i\vert^{\ell} \mathcal{Y}_{\ell}^m \left(\frac{\bold{x}-\bold{x}_i}{\vert \bold{x} - \bold{x}_i\vert}\right)$ is a homogenous, harmonic polynomial of degree $\ell$ in the variables $(\bold{x} - \bold{x}_i)_{\alpha}, ~ \alpha =1,2,3$. Consequently, the partial derivative $
	\partial_{\alpha}\left(\vert\bold{x}-\bold{x}_i\vert^{\ell} \mathcal{Y}_{\ell}^m \left(\frac{\bold{x}-\bold{x}_i}{\vert \bold{x} - \bold{x}_i\vert}\right)\right)$ must be a homogenous, harmonic polynomial of degree $\ell-1$ in $(\bold{x} - \bold{x}_i)_{\alpha}, ~ \alpha =1,2,3$. It follows that there exist coefficients $[d]_{\ell}^m, ~\ell \geq 1, ~-\ell \leq m \leq \ell$ such that the function $\psi^{\alpha}_i$ can be written as
	\begin{align*}
	\psi^{\alpha}_i(\bold{x})=\sum_{{\ell}=1}^{\infty}\sum_{m=-\ell}^{m=\ell} \frac{1}{r_i^{\ell}}[d]_{\ell}^m\vert\bold{x}-\bold{x}_i\vert^{\ell} \mathcal{Y}_{\ell}^m \left(\frac{\bold{x}-\bold{x}_i}{\vert \bold{x}-\bold{x}_i \vert}\right) \quad \forall \bold{x} \in \overline{\Omega_i},
	\end{align*}
	and therefore $\gamma_i^- \psi \in \breve{H}^{\frac{1}{2}}(\partial \Omega_i)$ as claimed. Recalling the definition of the local $||| \cdot |||_{i}$ norm, we see
	\begin{align}\nonumber
	||| \gamma_i^-\partial_{\alpha}\phi_{i,{\rm exc}}|||^2_{i} &= \Vert \mathbb{P}_{0, i} \gamma_i^-\partial_{\alpha}\phi_{i,{\rm exc}}\Vert_{L^2(\partial \Omega_i)}^2+ \left\langle \text{DtN}\gamma_i^-\psi^{\alpha}_i, \gamma_i^-\psi^{\alpha}_i\right\rangle_{\partial \Omega_i} \nonumber\\
	&= 4\pi r_i^2 C_{\rm const}^2 + \left\langle \text{DtN}\gamma_i^-\psi^{\alpha}_i, \gamma_i^-\psi^{\alpha}_i\right\rangle_{\partial \Omega_i}. \label{eq:Hassan4}
	\end{align}
	
	Our next task is to obtain a simple bound for the above duality pairing. Recall that by definition $\phi_{i,{\rm exc}}$ is harmonic on $\Omega_i$ and it therefore follows that any partial derivative $\partial_{\alpha} \phi_{i,{\rm exc}}, ~ \alpha=1,2,3$ is also harmonic in $\Omega_i$.  Consequently, Green's identity yields that
	\begin{align*}
	\vert\psi^{\alpha}_i\vert^2_{H^1(\Omega_i)}:=\int_{\Omega_i} \vert \nabla \psi^{\alpha}_i(\bold{x})\vert^2\, d\bold{x}
	=&\left\langle\text{DtN} \gamma_i^- \psi^{\alpha}_i, \gamma_i^- \psi^{\alpha}_i \right\rangle_{\partial \Omega_i},
	\end{align*}
	where $\vert \cdot \vert_{H^1(\Omega_i)}$ denotes the usual $H^1$ semi-norm on $\Omega_i$. Let $\vert \cdot \vert_{H^2(\Omega_i)}$ denote the usual $H^2$ semi-norm on $\Omega_i$. It is then clear that $\vert\partial_{\alpha}\phi_{i,{\rm exc}} \vert^2_{H^1(\Omega_i)} \leq \vert\phi_{i,{\rm exc}} \vert^2_{H^2(\Omega_i)}$. Furthermore, it is straightforward to show that there exists a constant $C_{r_i}$ depending only on the radius $r_i$ such that 
	\begin{align*}
	\left\langle \text{DtN} \gamma_i^- \partial_{\alpha}\phi_{i,{\rm exc}},\gamma_i^- \partial_{\alpha}\phi_{i,{\rm exc}} \right\rangle_{\partial \Omega_i}&= 	\vert\partial_{\alpha}\phi_{i,{\rm exc}} \vert^2_{H^1(\Omega_i)} 
	\leq \vert\phi_{i,{\rm exc}} \vert^2_{H^2(\Omega_i)}\leq C_{r_i} \big|\big|\big|\mathbb{P}^{\perp}_{0, i}\gamma_i^- \phi_{i,{\rm exc}}  \big|\big|\big|^2_{H^{\frac{3}{2}}(\partial \Omega_i)}.
	\end{align*}
	
	Using this bound in Equation \eqref{eq:Hassan4}, we obtain that
	\begin{align*}
	||| \gamma_i^-\partial_{\alpha}\phi_{i,{\rm exc}}|||^2_{i} \leq 4\pi r_i^2 C_{\rm const}^2 + C_{r_i} \big|\big|\big| \mathbb{P}^{\perp}_{0, i}\gamma_i^- \phi_{i,{\rm exc}}  \big|\big|\big|^2_{H^{\frac{3}{2}}(\partial \Omega_i)}.
	\end{align*}
	
	Moreover, since $C_{\rm const}$ depends on the radius $r_i$ and the coefficients $[\phi_i]_{1}^m$, $m \in \{-1, 0, 1\}$, we can deduce the existence of yet another constant $\widetilde{C_{r_i}}>0$ also depending only $r_i$ such that
	\begin{align*}
	||| \gamma_i^-\big(\boldsymbol{E}_{i}\big)_{\alpha}  |||^2_{i} \leq  \widetilde{C_{r_i}} \big|\big|\big| \mathbb{P}^{\perp}_{0, i}\gamma_i^- \phi_{i,{\rm exc}} \big|\big|\big|^2_{H^{\frac{3}{2}}(\partial \Omega_i)}.
	\end{align*}
	
	This completes the proof for Estimate \ref{eq:New_Ben_1}. We now proceed to the proof of Estimate \ref{eq:New_Ben_2}. To this end, let $i \in \{1, \ldots, N\}$ be fixed once again. We first define the function $\xi \in C^{\infty}(\overline{\Omega_i})$ as
	\begin{align*}
	\xi(\bold{x}):=\sum_{\ell=\ell_{\max} +2 }^{\infty}\sum_{m=-\ell}^{m=\ell} \frac{1}{r_i^{\ell}}[\phi_i]_{\ell}^m\,\partial_{\alpha} \left(\vert\bold{x}-\bold{x}_i\vert^{\ell}\mathcal{Y}_{\ell}^m \left(\frac{\bold{x}-\bold{x}_i}{\vert \bold{x}-\bold{x}_i\vert}\right)\right).
	\end{align*}
	
	It follows from Equation \eqref{eq:Hassan3} that for all $\bold{x}\in \overline{\Omega_i}$ we can write
	\begin{align*}
	\big(\partial_{\alpha}\phi_{i,{\rm exc}}\big)(\bold{x}) = \sum_{{\ell}=0}^{\ell_{\max} +1}\sum_{m=-\ell}^{m=\ell} \frac{1}{r_i^{\ell}}[\phi_i]_{\ell}^m\,\partial_{\alpha} \left(\vert\bold{x}-\bold{x}_i\vert^{\ell}\mathcal{Y}_{\ell}^m \left(\frac{\bold{x}-\bold{x}_i}{\vert \bold{x}-\bold{x}_i \vert}\right)\right)+\xi(\bold{x}).
	\end{align*}
	
	We now claim that $\mathbb{P}_{\ell_{\max}, i}\gamma_i^-\xi=0$. Indeed, as argued in the proof for Estimate \eqref{eq:New_Ben_1}, the function $\vert\bold{x}-\bold{x}_i\vert^{\ell} \mathcal{Y}_{\ell}^m \left(\frac{\bold{x}-\bold{x}_i}{\vert \bold{x} - \bold{x}_i\vert}\right)$ is a homogenous, harmonic polynomial of degree~$\ell$ in the variables $(\bold{x} - \bold{x}_i)_{\alpha}, ~ \alpha =1,2,3$. This implies that for any $\ell\geq 1$ the partial derivative $\partial_{\alpha}\left(\vert\bold{x}-\bold{x}_i\vert^{\ell} \mathcal{Y}_{\ell}^m \left(\frac{\bold{x}-\bold{x}_i}{\vert \bold{x} - \bold{x}_i\vert}\right)\right)$ is a homogenous, harmonic polynomial of degree $\ell-1$ in $(\bold{x} - \bold{x}_i)_{\alpha}, ~ \alpha =1,2,3$. Consequently, we obtain the existence of coefficients $[d]_{\ell}^m, ~\ell \geq 1, ~-\ell \leq m \leq \ell$ such that the function $\xi$ can be written as
	\begin{align*}
	\xi(\bold{x})=\sum_{{\ell}=\ell_{\max}+1}^{\infty}\sum_{m=-\ell}^{m=\ell} \frac{1}{r_i^{\ell}}[d]_{\ell}^m\vert\bold{x}-\bold{x}_i\vert^{\ell} \mathcal{Y}_{\ell}^m \left(\frac{\bold{x}-\bold{x}_i}{\vert \bold{x}-\bold{x}_i \vert}\right) \quad \forall \bold{x} \in \overline{\Omega_i},
	\end{align*} 
	and therefore $\mathbb{P}_{\ell_{\max}, i}\gamma_i^- \xi=0$ as claimed.
	
	The remainder of the proof is now essentially identical to the proof for Estimate \ref{eq:New_Ben_1} with some obvious changes. Indeed, repeating the arguments we presented previously, we arrive at the inequality:
	\begin{align*}
	||| \mathbb{P}_{\ell_{\max}, i}\gamma_i^-\big(\boldsymbol{E}_{i}\big)_{\alpha}  |||^2_{i} \leq  \widetilde{C_{r_i}} \big|\big|\big|\mathbb{P}_{\ell_{\max}+1, i}\mathbb{P}^{\perp}_{0, i}\gamma_i^- \phi_{i,{\rm exc}}  \big|\big|\big|^2_{H^{\frac{3}{2}}(\partial \Omega_i)},
	\end{align*}
	which completes the proof for Estimate \ref{eq:New_Ben_2}.
\end{proof}

Lemma \ref{lem:hassan} is the only tool required to prove Theorem \ref{lem:hassan4} on the approximability of the electrostatic forces. For the sake of brevity however-- and owing to the fact that Theorem \ref{thm:Force} is not the main focus of our analysis, we do not provide the proof. The interested reader can find a detailed proof in \cite[Section 5.1.3, Theorem 5.18]{Hassan_Dis}. In order to prove Theorem \ref{thm:Force}, we require an additional result, which follows as a straightforward corollary of Lemma \ref{lem:hassan} in the special case when the charge distribution $\sigma $ satisfies $\sigma \in H^{s}(\partial \Omega)$ for some $s \geq \frac{1}{2}$.

\begin{corollary}[Estimates in terms of Induced Surface Charges]\label{lem:hassan3}~
	\noindent Let $s \geq \frac{1}{2}$, let $\sigma \in H^{s}(\partial \Omega)$ be a given charge distribution and for each $i \in \{1, \ldots, N\}$ let $\boldsymbol{E}_{i} \in L_{\rm loc}^2(\mathbb{R}^3; \mathbb{R}^3)$ denote the $i$ excluded electric field generated by $\sigma$ as defined through Definition \ref{def:field2}. Then there exists a constant ${C_{\rm field}}> 0$ that depends only on the radii $\{r_j\}_{i=1}^N$ of the open balls $\{\Omega_j\}_{i=1}^N$ such that for each $\alpha =1,2,3$ the following hold:
	\begin{align}\label{eq:Estimate1}
	\sum_{i=1}^N \big|\big|\big|\gamma^-_i \big(\boldsymbol{E}_{i}\big)_{\alpha}\big|\big|\big|_i^2 &\leq C_{\rm field} \left( \big|\big|\big|\mathbb{P}_0^{\perp}\mathcal{V}\sigma\big|\big|\big|^2_{H^{\frac{3}{2}}(\partial \Omega)}+ \big|\big|\big|\mathbb{Q}_0^{\perp}\sigma \big|\big|\big|^2\right),\\ 
	\sum_{i=1}^N \big|\big|\big|\mathbb{P}_{\ell_{\max}, i}\gamma^-_i \big(\boldsymbol{E}_{i}\big)_{\alpha} \big|\big|\big|_i^2 &\leq C_{\rm field}\left(\frac{\ell_{\max}+1}{\min r_i}\right)^2\left( \big|\big|\big|\mathbb{P}_{\ell_{\max}+1}\mathbb{P}_0^{\perp}\mathcal{V}\sigma\big|\big|\big|^2+ \big(\big|\big|\big|\mathbb{Q}_{\ell_{\max}+1}\mathbb{Q}_0^{\perp}\sigma \big|\big|\big|^*\big)^2\right).		\label{eq:Estimate2}
	\end{align}
\end{corollary}
\begin{proof}
	The proof for the estimate \eqref{eq:Estimate1} is straightforward and relies primarily on the fact that for all $i \in \{1, \ldots, N\}$ the $i$ excluded electric potential satisfies
	\begin{align*}
	\gamma_i^-\phi_{i, {\rm exc}} = \mathcal{V}\big(\sigma- \widetilde{\sigma_i}\big)=\mathcal{V}\sigma-\mathcal{V}\widetilde{\sigma_i},
	\end{align*}
	where $\widetilde{\sigma_i}:=\sigma = \sigma_{i, \rm exc}$ is the `local' charge distribution supported only on the sphere $\partial \Omega_i$. Using the fact that the spherical harmonics are eigenfunctions of the single layer boundary operator on the sphere (see, e.g., \cite{vico2014boundary}) then allows us to obtain appropriate bounds for the second term in the above decomposition. The estimate \eqref{eq:Estimate2} can be deduced by repeating the above arguments and using the equivalence of norms on the space $W^{\ell_{\max}}$. 

\end{proof}

We are now ready to state the proof of Theorem \ref{thm:Force}.

\begin{proof}[{{Proof of Theorem}} \ref{thm:Force}:]
	
	Let $i \in \{1, \ldots, N\}$ be fixed. Let $\boldsymbol{E}_{i}, \boldsymbol{E}^{\ell_{\max}}_{i} \in L_{\rm loc}^2(\mathbb{R}^3; \mathbb{R}^3)$ denote the $i$ excluded electric fields generated by the charge distributions $\nu$ and $\nu_{\ell_{\max}}$ respectively as defined through Definition \ref{def:field2}, and let $\nu_i:= \nu\vert_{\partial \Omega_i}$ and $\nu_i^{\ell_{\max}}:=\nu_{\ell_{\max}} \vert_{\partial \Omega_i}$. Arguing exactly as in the proof of Theorem \ref{lem:hassan4}, we see that for each $\alpha =1, 2, 3$ it holds that 
	\begin{align*}
	\Big\vert (\boldsymbol{F}_i)_{\alpha} - \big(\boldsymbol{F}_i^{\ell_{\max}}\big)_{\alpha}\Big \vert &\leq 
	|||\nu_i- \nu_i^{\ell_{\max}} |||_i^* ||| \gamma_i^-\big(\boldsymbol{E}_i\big)_{\alpha}|||_i+|||\nu_i^{\ell_{\max}} |||_i^* ||| \mathbb{P}_{\ell_{\max}, i}\gamma_i^-\big(\boldsymbol{E}_i-\boldsymbol{E}_i^{\ell_{\max}}\big)_{\alpha}|||_i.
	\end{align*}
	
	Consequently, using the Cauchy-Schwarz inequality we have
	\begin{align*}
	\sum_{i=1}^N\Big\vert (\boldsymbol{F}_i)_{\alpha} - \big(\boldsymbol{F}_i^{\ell_{\max}}\big)_{\alpha}\Big \vert &\leq {|||\nu- \nu_{\ell_{\max}} |||^*} \left(\sum_{i=1}^N||| \gamma_i^-\big(\boldsymbol{E}_i\big)_{\alpha}|||^2_i\right)^{\frac{1}{2}}\\
	&+{|||\nu_{\ell_{\max}} |||^*} \left(\sum_{i=1}^N||| \mathbb{P}_{\ell_{\max}, i}\gamma_i^-\big(\boldsymbol{E}_i-\boldsymbol{E}_i^{\ell_{\max}}\big)_{\alpha}|||_i^2 \right)^{\frac{1}{2}}.
	\end{align*}
	
	Next, using the linearity of the underlying operators and applying Corollary \ref{lem:hassan3}, which is now applicable since $s \geq \frac{1}{2}$ we obtain that 
	\begin{multline}
	\sum_{i=1}^N\Big\vert (\boldsymbol{F}_i)_{\alpha} - \big(\boldsymbol{F}_i^{\ell_{\max}}\big)_{\alpha}\Big \vert \leq C^{\frac{1}{2}}_{\rm field}  {|||\nu- \nu_{\ell_{\max}} |||^*} \left(|||\mathbb{P}_0^{\perp}\mathcal{V}\nu|||^2_{H^{\frac{3}{2}}(\partial \Omega)}+ ||| \mathbb{Q}_0^{\perp}\nu |||^2\right)^{\frac{1}{2}}\\ 
	+C^{\frac{1}{2}}_{\rm field} {|||\nu_{\ell_{\max}} |||^*} \left(\frac{\ell_{\max}+1}{\min r_i}\right)\bigg(||| \mathbb{P}_{\ell_{\max}+1}\mathbb{P}^{\perp}_{0}\mathcal{V}(\nu-\nu_{\ell_{\max}})|||^2+{|||\mathbb{Q}_{\ell_{\max}+1}\mathbb{Q}_0^{\perp}\big(\nu-\nu_{\ell_{\max}}\big) |||^*}^2\bigg)^{\frac{1}{2}}. \label{eq:simple0}
	\end{multline}
	
	In order to simplify Inequality \eqref{eq:simple0}, we first use the triangle inequality to obtain 
	\begin{align}\label{eq:simple1}
	|||\nu_{\ell_{\max}} |||^* \leq ||| \nu|||^* + |||\nu-\nu_{\ell_{\max}} |||^*.
	\end{align}
	
	Next, for ease of exposition, let us define the terms
	\begin{align*}
	{\rm (III)}:=&  ~|||\mathbb{P}_0^{\perp}\mathcal{V}\nu|||^2_{H^{\frac{3}{2}}(\partial \Omega)}+ ||| \mathbb{Q}_0^{\perp}\nu |||^2,\\
	{\rm (IV)}:=&~||| \mathbb{P}_{\ell_{\max}+1}\mathbb{P}^{\perp}_{0}\mathcal{V}(\nu-\nu_{\ell_{\max}})|||^2+{|||\mathbb{Q}_{\ell_{\max}+1}\mathbb{Q}_0^{\perp}\big(\nu-\nu_{\ell_{\max}}\big) |||^*}^2.
	\end{align*}
	
	The term (III) can be simplified by observing that the BIE \eqref{eq:3.3a} implies that
	\begin{align*}
	\mathbb{P}_0^{\perp}\mathcal{V}\nu= \frac{\kappa_0}{\kappa_0-\kappa}\text{DtN}^{-1}\left(\mathbb{Q}_0^{\perp}\nu-\frac{4\pi}{\kappa_0}\mathbb{Q}_0^{\perp}\sigma_f\right).
	\end{align*}
	Thus, using the fact that $||| \text{DtN}^{-1}\mathbb{Q}_0^{\perp}\sigma|||^2_{H^{\frac{3}{2}}(\partial \Omega)}= |||\mathbb{Q}_0^{\perp}\sigma|||^2$ for any $\sigma \in H^{-\frac{1}{2}}(\partial \Omega)$ (c.f., Definition~\ref{def:7.2} of the higher order norms), we can conclude that there exists a constant $C_{\rm diel}$ depending only on the dielectric constants such that 
	\begin{align}\label{eq:simple2}
	{\rm (III)} \leq C^2_{\rm diel}\left({|||\mathbb{Q}_0^{\perp} \nu|||}^2+ {||| \mathbb{Q}_0^{\perp}\sigma_f|||}^2\right). 
	\end{align}
	
	In order to simplify (IV), we observe that the BIE \eqref{eq:3.3a} and the Galerkin discretisation \eqref{eq:Galerkina} together imply that $\mathbb{Q}_0 (\nu-\nu_{\ell_{\max}})=0$. Consequently, there exists a function $\zeta \in \breve{H}^{\frac{1}{2}}(\partial \Omega)$ such that $\nu-\nu_{\ell_{\max}}:= \text{DtN}\zeta$. We therefore have
	\begin{align*}
	||| \mathbb{P}_{\ell_{\max}+1}\mathbb{P}^{\perp}_{0}\mathcal{V}(\nu-\nu_{\ell_{\max}})|||^2 \leq ||| \mathbb{P}^{\perp}_{0}\mathcal{V}\text{DtN}\zeta|||^2\leq \frac{c_{\rm equiv}^2}{c_{\mathcal{V}}} ||| \zeta|||^2=\frac{c_{\rm equiv}^2}{c_{\mathcal{V}}} {|||\text{DtN}\zeta|||^*}^2,
	\end{align*}
	where the second inequality can be deduced from the bound \eqref{eq:contraction} in Lemma \ref{lem:single}. Thus,
	\begin{align}\label{eq:simple3}
	{\rm (IV)}\leq \left(1 + \frac{c_{\rm equiv}^2}{c_{\mathcal{V}}}\right){|||\nu-\nu_{\ell_{\max}} |||^*}^2
	\end{align}
	
	Using the bounds \eqref{eq:simple1}-\eqref{eq:simple3}, we can simplify the estimate \eqref{eq:simple0} to obtain 
	\begin{equation}\label{eq:simple0_next}
	\begin{split}
	\sum_{i=1}^N\Big\vert (\boldsymbol{F}_i)_{\alpha} - \big(\boldsymbol{F}_i^{\ell_{\max}}\big)_{\alpha}\Big \vert &\leq C^{\frac{1}{2}}_{\rm field}C_{\rm diel} ||| \nu - \nu_{\ell_{\max}}|||^* \Big(|||\mathbb{Q}_0^{\perp}\nu|||+||| \mathbb{Q}_0^{\perp}\sigma_f|||\Big)\\
	&+C^{\frac{1}{2}}_{\rm field} ||| \nu |||^* \left(\frac{\ell_{\max}+1}{\min r_i}\right) \left(1+ \frac{c_{\rm equiv}}{\sqrt{c_{\mathcal{V}}}}\right)||| \nu - \nu_{\ell_{\max}}|||^*\\
	&+C^{\frac{1}{2}}_{\rm field} \left(\frac{\ell_{\max}+1}{\min r_i}\right) \left(1+ \frac{c_{\rm equiv}}{\sqrt{c_{\mathcal{V}}}}\right){||| \nu - \nu_{\ell_{\max}}|||^*}^2.
	\end{split}
	\end{equation}
	
	In principle, the next step is to simplify further the bound \eqref{eq:simple0_next} using the convergence rates from Theorem \ref{thm:well-posed}. However, in order to obtain a succinct final result, we first rewrite some terms in the estimate \eqref{eq:simple0_next}. Using basic calculus and Definition \ref{def:7.2} of the higher order norms we write for any $s \geq \frac{1}{2}$:
	\begin{equation} \label{eq:simpl0_next1}
	\begin{split}
	|||\mathbb{Q}_0^{\perp}\nu|||+||| \mathbb{Q}_0^{\perp}\sigma_f||| &\leq (\max r_i)^{s-\frac{1}{2}} \left( ||| \nu|||_{H^s(\partial \Omega)}+||| \sigma_f|||_{H^s(\partial \Omega)}\right), \quad \text{and}\\
	||| \nu |||^* \leq ||| \nu |||^* + ||| \sigma_f|||^* &\leq \big(1+(\max r_i)^{s+\frac{1}{2}}\big)\left(|||\nu |||_{{H}^{s}(\partial \Omega)} + |||\sigma_f|||_{{H}^{s}(\partial \Omega)}\right).
	\end{split}
	\end{equation}
	
	Using the bounds \eqref{eq:simpl0_next1} together with the convergence rates from Theorem~\ref{thm:well-posed}, we can simplify the estimate \eqref{eq:simple0_next} to obtain that
	\begin{multline*}
	\hspace{-4mm}\sum_{i=1}^N\Big\vert (\boldsymbol{F}_i)_{\alpha} - \big(\boldsymbol{F}_i^{\ell_{\max}}\big)_{\alpha}\Big \vert \leq C^{\frac{1}{2}}_{\rm field}\bigg(|||\nu|||_{H^s(\partial \Omega)}+|||\sigma_f|||_{H^s(\partial \Omega)}\bigg)^2
	\\\Bigg( {C_{\rm charges}C_{\rm diel}}(\max r_i)^{s-\frac{1}{2}}\left(\frac{\max r_i}{\ell_{\max}+1}\right)^{s+\frac{1}{2}}\\
	+C_{\rm charges}\frac{\max r_i +(\max r_i)^{s+\frac{3}{2}}}{\min r_i}\left(1+ \frac{c_{\rm equiv}}{\sqrt{c_{\mathcal{V}}}}\right) \left(\frac{\max r_i}{\ell_{\max}+1}\right)^{s-\frac{1}{2}}\\
	+C^2_{\rm charges} \frac{\max r_i}{\min r_i}\left(1+ \frac{c_{\rm equiv}}{\sqrt{c_{\mathcal{V}}}}\right) \left(\frac{\max r_i}{\ell_{\max}+1}\right)^{2s}\Bigg).
	\end{multline*}
	
	The proof now follows by defining the constant $C_{\rm force}>0$ as
	\begin{align*}
	C_{\rm force}:= C^{\frac{1}{2}}_{\rm field} \max\Bigg \{{C_{\rm charges}C_{\rm diel}}(\max r_i)^{s-\frac{1}{2}}, ~ C_{\rm charges}^2\frac{\max r_i}{\min r_i}\left(1+ \frac{c_{\rm equiv}}{\sqrt{c_{\mathcal{V}}}}\right),\\  ~ C_{\rm charges}\frac{\max r_i +(\max r_i)^{s+\frac{3}{2}}}{\min r_i}\left(1+ \frac{c_{\rm equiv}}{\sqrt{c_{\mathcal{V}}}}\right) \Bigg\},
	\end{align*}
	and using the fact that for $s\geq \frac{1}{2}$ and $\ell_{\max}+1 > \max_{j=1, \ldots, N} r_j$ it holds that
	\begin{align*}
	\left(\frac{\max r_i}{\ell_{\max}+1}\right)^{2s}\leq\left(\frac{\max r_i}{\ell_{\max}+1}\right)^{\frac{1}{2}+s}\leq	\left(\frac{\max r_i}{\ell_{\max}+1}\right)^{-\frac{1}{2}+s}. 
	\end{align*}
\end{proof}

\vspace{-5mm}
\section{Solution Strategy and Numerical Results}\label{sec:5}
The goal of this section is two-fold. First, we present a linear scaling in complexity solution strategy for computing the approximate electrostatic forces $\{\boldsymbol{F}_i^{\ell_{\max}}\}_{i=1}^N$ defined through Definition \ref{def:Force2_approx}. Second, we provide numerical evidence that supports our theoretical results in Section \ref{sec:4a} as well as our claim that the electrostatic forces can be computed with linear scaling (in $N$) computational cost. In the sequel, we assume the setting of Sections \ref{sec:2}, \ref{sec:3} and \ref{sec:4}.

\subsection{Computing the Electrostatic Forces}\label{sec:5a}
{In view of the results and discussion presented in Sections \ref{sec:3} and \ref{sec:4}, the first step in the computation of the approximate electrostatic forces is obtaining the solution $\nu_{\ell_{\max}} \in W^{\ell_{\max}}$ to the Galerkin discretisation \eqref{eq:Galerkina}. Consequently, if we wish to obtain a linear scaling in complexity solution strategy for the computation of the approximate electrostatic forces, we \emph{must} possess a linear scaling in complexity strategy for calculating the approximate induced surface charge $\nu_{\ell_{\max}}$. Such a strategy was discussed in detail in the contribution \cite{Hassan2}. Theorem \ref{thm:iterations} in Section \ref{sec:2} of the current article summarises an important result from \cite{Hassan2} and states that one can use a GMRES-based solution strategy to obtain an approximation $\nu_{\ell_{\max}}^{\rm approx} \in W^{\ell_{\max}}$-- up to a given tolerance-- of $\nu_{\ell_{\max}} $ using only $\mathcal{O}(N)$ operations.
	
	Consequently, in practice we typically compute-- up to a required tolerance $\epsilon$-- an approximation $\nu_{\ell_{\max}}^{\rm approx} \in W^{\ell_{\max}}$ of the solution $\nu_{\ell_{\max}} \in W^{\ell_{\max}}$ to the Galerkin discretisation \eqref{eq:Galerkina}, and we use $\nu_{\ell_{\max}}^{\rm approx}$ rather than $\nu_{\ell_{\max}}$ to calculate the approximate electrostatic forces as defined through Definition \ref{def:Force2_approx}. It is therefore important to obtain stability estimates for the approximate forces derived from the approximation $\nu_{\ell_{\max}}^{\rm approx}$. To this end, we have the following result.
	
	{\begin{lemma}[Stability of Forces with Respect to Linear Solver Tolerance]~\label{lem:stable}
			\noindent Let $\epsilon > 0$ and $\ell_{\max} \in \mathbb{N}$, let $\sigma_f \in H^{-\frac{1}{2}}(\partial \Omega)$ be a given free charge, let $\nu_{\ell_{\max}} \in W^{\ell_{\max}}$ be the unique solution to the Galerkin discretisation \eqref{eq:Galerkina}, let $\nu_{\ell_{\max}}^{\rm approx} \in W^{\ell_{\max}}$ be an approximation to $\nu_{\ell_{\max}}$ with relative tolerance $\epsilon \ll 1$ as described in Theorem \ref{thm:iterations}, and for each $i=1, \ldots, N$, let $\boldsymbol{F}_i^{\ell_{\max}}$ and $\widehat{\boldsymbol{F}}_i^{\ell_{\max}}$ denote the approximate net force acting on the dielectric particle $\Omega_i$, generated by the charge distributions $\nu_{\ell_{\max}}$ and $\nu_{\ell_{\max}}^{\rm approx}$ respectively as defined through Definition \ref{def:Force2_approx}. Then there exists a constant $C_{\rm stability}>0$ that depends on $\ell_{\max}$, the dielectric constants, the radii of the open balls and the minimum inter-sphere separation distance but is independent of the number $N$ of dielectric particles such that
			\begin{align}\label{eq:stable}
			\frac{\sum_{i=1}^N \sum_{\alpha=1}^3\Big\vert (\boldsymbol{F}^{\ell_{\max}}_i)_{\alpha} - \big(\widehat{\boldsymbol{F}}_i^{\ell_{\max}}\big)_{\alpha}\Big \vert}{||| \mathbb{Q}_0^{\perp}\nu_{\ell_{\max}}|||^*  + ||| \mathbb{Q}_0^{\perp}\nu|||^* + ||| \mathbb{Q}_0^{\perp}\sigma_f|||^*} \leq \epsilon\, C_{\rm stability} \left(|||\mathbb{Q}_0^{\perp} \nu_{\ell_{\max}}|||+ ||| \mathbb{Q}_0^{\perp}\sigma_f|||\right).
			\end{align}
		\end{lemma}
		\begin{proof}
			The proof of Lemma \ref{lem:stable} uses arguments similar to those stated in the proof of our main result Theorem \ref{thm:Force} together with the estimate given by Theorem \ref{thm:iterations} from Section \ref{sec:2c}. For the sake of brevity, we omit a formal argument here but a precise proof can be found in \cite[Section 5.3.1, Lemma 5.34]{Hassan_Dis}.
		\end{proof}

		\begin{remark}
			{Consider Lemma \ref{lem:stable}. Essentially, this result states that if one uses an approximation $\nu_{\ell_{\max}}^{\rm approx}\in W^{\ell_{\max}}$ to the true solution $\nu_{\ell_{\max}} \in W^{\ell_{\max}}$ of the Galerkin discretisation \eqref{eq:Galerkina}, with relative tolerance $\epsilon$ as detailed in Theorem \ref{thm:iterations}, to compute the approximate electrostatic forces, then the relative error in these forces (with respect to the true approximate electrostatic forces) is bounded by $\epsilon$ times the constant $C_{\rm stability}$ which does not explicitly depend on $N$. Since the tolerance $\epsilon$ can be controlled by modifying the linear solver tolerance used when computing $\nu_{\ell_{\max}}^{\rm approx}$, it follows that for any geometrical configuration in the family of geometries $\{\Omega_{\mathcal{F}}\}_{\mathcal{F} \in \mathcal{I}}$ satisfying {\textbf{ A1)-A3)}}, this relative error in the forces can be made arbitrarily small \underline{independent of the number of dielectric spheres $N_{\mathcal{F}}$.}}
		\end{remark}
		
		\noindent	We are now ready to state our solution strategy for computing the approximate electrostatic forces. Given a known free charge $\sigma_f \in {H}^{-\frac{1}{2}}(\partial \Omega)$, the goal is to obtain for each $i \in \{1, \ldots, N\}$, the approximate net electrostatic force $\boldsymbol{F}_i^{\ell_{\max}} \in \mathbb{R}^3$ acting on the dielectric particle represented by $\Omega_i$.

		\begin{enumerate}
			\item[Step 1:] Fix $\ell_{\max} \in \mathbb{N}$ and compute-- up to some fixed tolerance-- the approximate solution $\nu_{\ell_{\max}}^{\rm approx} \in W^{\ell_{\max}}$ to the Galerkin discretisation \eqref{eq:Galerkina}, thereby obtaining the approximate local spherical harmonic expansion coefficients of $\nu^{\rm approx}_{\ell_{\max}}$ on the spheres $\{\partial \Omega_i\}_{i=1}^N$. This computation can be done according to the solution strategy presented in \cite{Hassan2}. In view of Theorem \ref{thm:iterations}, (see also \cite{Hassan2}) and Lemma \ref{lem:stable}, the computational cost of obtaining $\nu^{\rm approx}_{\ell_{\max}}$ with a fixed and given error tolerance is $\mathcal{O}(N)$ for any geometrical configuration belonging to the family of geometries $\{\Omega_{\mathcal{F}}\}_{\mathcal{F} \in \mathcal{I}}$ satisfying {\textbf{ A1)-A3)}}.
			
			\item[Step 2:] Compute $\lambda^{\rm approx}_{\ell_{\max}}:=\mathbb{P}_{\ell_{\max}+1}\mathcal{V}\nu^{\rm approx}_{\ell_{\max}}$. This gives access to the local spherical harmonic expansion coefficients of $\lambda^{\rm approx}_{\ell_{\max}}$ up to order $\ell_{\max}+1$ on the spheres $\{\partial \Omega_i\}_{i=1}^N$. Notice that in view of Definition \ref{def:Force2_approx} of the approximate electrostatic force and proof of Lemma \ref{lem:hassan}, we require only the expansion coefficients up to order $\ell_{\max}+1$. Due to the use of the FMM, the computational cost of this step is also $\mathcal{O}(N)$.
			
			\item[Step 3:] Compute for each $\alpha=1,2,3$ and all $1\leq \ell \leq \ell_{\max}+1$, $-\ell \leq m \leq \ell$ the partial derivatives
			\begin{align*}
			\partial_{\alpha} \left(\vert \bold{x}\vert^{\ell} \mathcal{Y}_{\ell}^m\left(\frac{\bold{x}}{\vert \bold{x}\vert}\right)\right), ~ ~\text{where } \bold{x} \in \mathbb{S}^2.
			\end{align*}
			
			These derivatives can be computed analytically so the computational cost of this step is $\mathcal{O}(1)$.
			
			
			\item[Step 4:] Using the expansion coefficients from Steps 1 and 2, the partial derivatives from Step 3, and the representation of $\partial_{\alpha}\phi^{\rm approx}_{i, {\rm exc}} $ given by Equation \eqref{eq:Hassan3}, compute for each $\alpha=1,2,3$ and $i=1, \ldots, N$, the trace $ \mathbb{P}_{\ell_{\max}, i}\gamma_i^-\big(\boldsymbol{E}^{\rm approx}_i\big)_{\alpha}$ of the approximate $i$ excluded electric field. This step requires $\mathcal{O}(N)$ operations.
			
			\item[Step 5:] The approximate electrostatic forces $\{\widehat{\boldsymbol{F}}^{\ell_{\max}}\}_{i=1}^N$ acting on the dielectric particles represented by $\{\Omega_i\}_{i=1}^N$ can then be obtained by computing the integrals
			\begin{align*}
			\widehat{\boldsymbol{F}}_i^{\ell_{\max}}= \kappa_0\int_{\partial \Omega_i}\nu^{\rm approx}_{\ell_{\max}} (\bold{x})\big( \mathbb{P}_{\ell_{\max}, i}\gamma_i^-\boldsymbol{E}^{\rm approx}_i\big)(\bold{x})\, d\bold{x}, \quad i=1, \ldots, N.
			\end{align*}
			
			The computational cost of this step is also $\mathcal{O}(N)$. 
			
		\end{enumerate}
	}

	{ We conclude this subsection by emphasising once again the key implication of Theorem \ref{thm:Force} and the solution strategy stated above. Given a geometrical configuration belonging to the family of geometries $\{\Omega_{\mathcal{F}}\}_{\mathcal{F} \in \mathcal{I}}$ satisfying {\textbf{ A1)-A3)}}, which consists of a system of $N_{\mathcal{F}}$ interacting dielectric particles, we can compute- up to any given error tolerance-- the electrostatic forces acting on each spherical dielectric particle in $\mathcal{O}(N_{\mathcal{F}})$ operations. In other words our numerical method for computing the forces is linear scaling in cost. Since Theorem \ref{thm:Force} yields $N$-independent error estimates for the electrostatic forces, the method is also $N$-error stable. We can therefore conclude that under the geometrical assumptions {\textbf{ A1)-A3)}}, the numerical method for obtaining the electrostatic forces described in this contribution is indeed \emph{linear scaling in accuracy}, i.e., the computational cost of obtaining the approximate forces up to a fixed relative error scales linearly in $N$.}
	
	\subsection{Numerical Experiments}\label{sec:5b}
	As mentioned in the solution strategy in Section \ref{sec:5a}, we use the FMM to compute matrix-vector products involving the global single layer boundary operator $\mathcal{V}$. This allows us to achieve the required linear scaling in cost albeit, at the cost of introducing a controllable FMM approximation error. Standard FMM libraries usually accept only point charges as inputs. To suit our needs, we have used instead a modification of the ScalFMM library (see \cite{lindgren2018} for an explanation of the modification and \cite{agullo2014task,blanchard2015scalfmm} for details on the ScalFMM library). Additionally, we have used the Krylov subspace solver GMRES (see, e.g., \cite{Saad,saad1986gmres}) to solve all underlying linear systems. In the sequel, the numerical tests one through three, which were designed to test the accuracy of our numerical algorithm were performed using a single level FMM octree and with the GMRES tolerance set to $10^{-11}$. This prevents the introduction of the FMM approximation error and linear solver error respectively. 
	
	\begin{figure}[h!]
		\centering
		\begin{subfigure}[t]{0.4\textwidth}
			\centering
			\includegraphics[width=\textwidth]{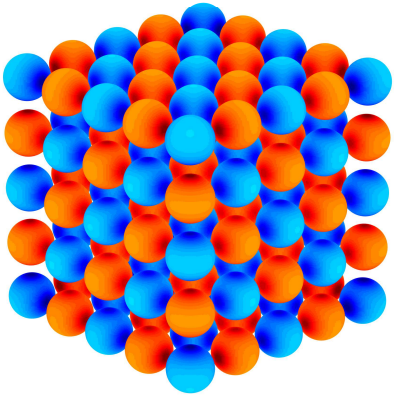} 
			\caption{Dielectric spheres with unit positive or negative charge, arranged in an alternating fashion on a three dimensional, regular cubic lattice.}
			\label{fig:11}
		\end{subfigure}\hfill
		\begin{subfigure}[t]{0.4\textwidth}
			\centering
			\includegraphics[width=\textwidth]{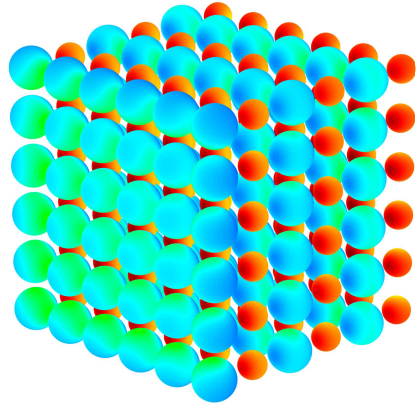} 
			\caption{Dielectric spheres with unit positive or negative charge, arranged in alternating layers on a three dimensional, regular cubic lattice. }
			\label{fig:12}
		\end{subfigure}
		\caption{The two basic geometric settings we use for the majority of our numerical experiments. Different colours indicate the degree of polarisation; red indicates positive and blue indicates negative charge.}
	\end{figure}
	
	With one exception, the numerical experiments in this section were performed on two basic geometrical settings. Both settings consist of the same two types of dielectric spherical particles, the first with radius 3, dielectric constant 10 and carrying unit negative free charge and the second with radius 2, dielectric constant 5 and carrying unit positive free charge, arranged on a regular cubic lattice of edge length $E$. In the first setting however, the lattice is organised such that the positive and negatively charged particles are arranged in an alternating fashion and the edge length $E$ is set to 6. In contrast, the lattice in the second setting is organised such that like-charged particles are arranged in layers, and we set $E=7$. Figures \ref{fig:11} and \ref{fig:12} display the first and second types of lattice structures respectively. In both cases, the external medium is assumed to be vacuum, i.e.,~$\kappa_0=1$. The total number of dielectric particles will typically vary from experiment to experiment.\\

	\noindent {\textbf{Test 1: Exponential Convergence}}~
	Our first set of numerical experiments is designed to demonstrate the exponential convergence of the approximate electrostatic forces. We set the total number $N$ of dielectric spherical particles to be $125$ in the case of the first lattice (see Figure \ref{fig:11}) and 216 in the case of the second lattice (see Figure \ref{fig:12}) and compute the average error in the approximate forces $\{\boldsymbol{F}^{\ell_{\max}}_{i}\}_{i=1}^N$  for different values of $\ell_{\max}$. The results are displayed in Figures \ref{fig:21} and \ref{fig:22} for the first and second lattice respectively. The reference forces $\{\boldsymbol{F}_i\}_{i=1}^N$ were obtained from the reference solution $\nu$ to the BIE \eqref{eq:3.3a}, which was computed by setting the discretisation parameter $\ell_{\max}=20$. For comparison, we have also plotted the average error in the approximate induced surface charge $\nu_{\ell_{\max}}$.
	
	The numerical results displayed in Figures \ref{fig:21} and \ref{fig:22} have three key features of interest. First, we observe the exponential convergence of the approximate forces predicted by Corollary \ref{thm:exp}. Second, we see that the rate of convergence is slower in the case of the first lattice which has a smaller edge length $E$. This is in agreement with our theoretical results as we explain in the next set of numerical experiments. {Finally, we observe that the convergence rates for the forces are nearly twice those (in exponential terms) of the induced surface charge. This agrees with the well-known phenomenon of the doubling of the convergence rates for linear functionals, which can be demonstrated through the so-called Aubin-Nitsche duality technique. Unfortunately, using such a duality trick leads to convergence rates for the electrostatic forces that cannot be shown to be independent of $N$, and we have therefore not pursued this approach.}
	
	\begin{figure}[h!]
		\centering
		\begin{subfigure}[t]{0.48\textwidth}
			\centering
			\includegraphics[width=\textwidth]{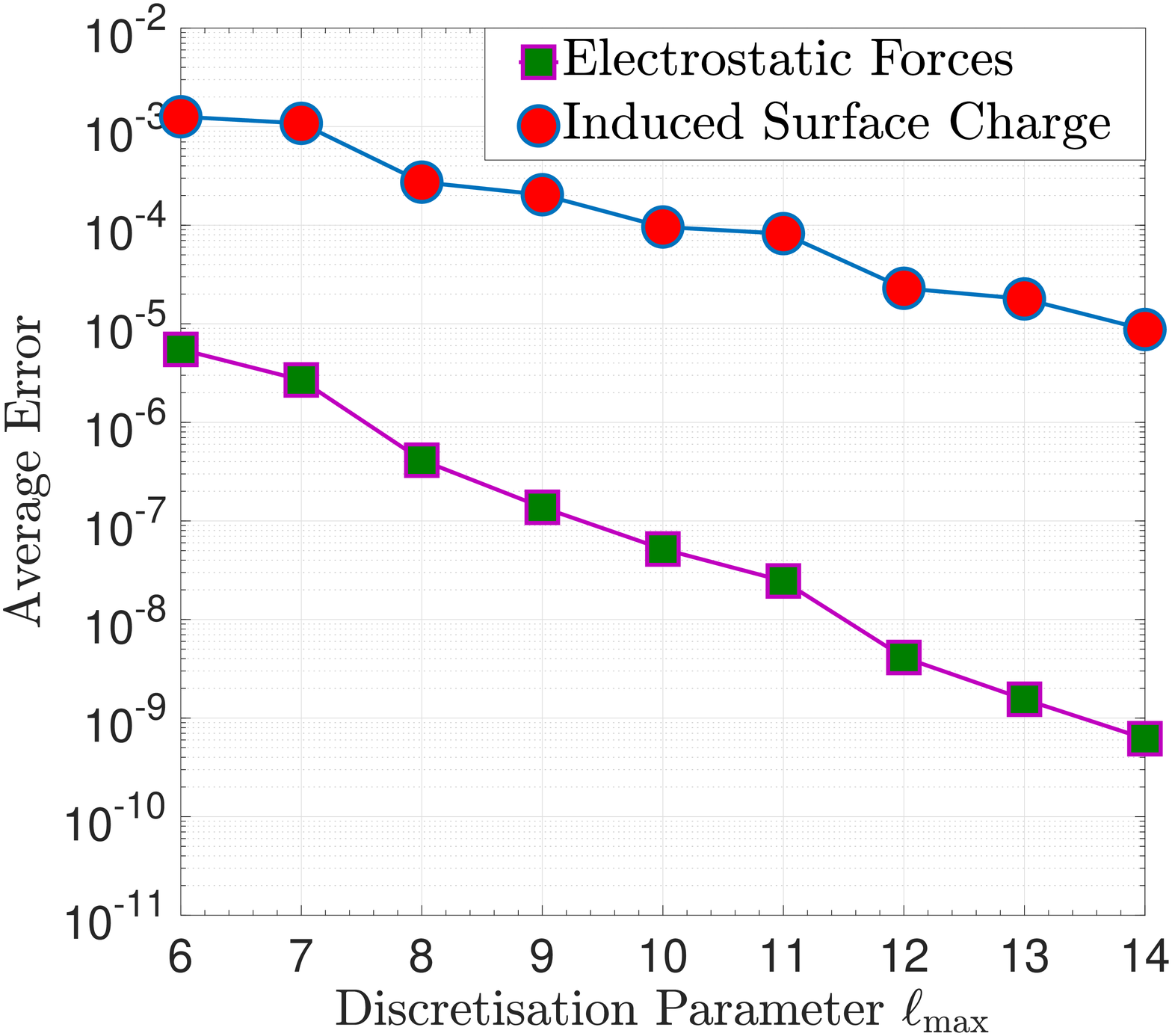} 
			\caption{Numerical results for the first type of lattice structure (Figure \ref{fig:11}). }
			\label{fig:21}
		\end{subfigure}\hfill
		\begin{subfigure}[t]{0.48\textwidth}
			\centering
			\includegraphics[width=\textwidth]{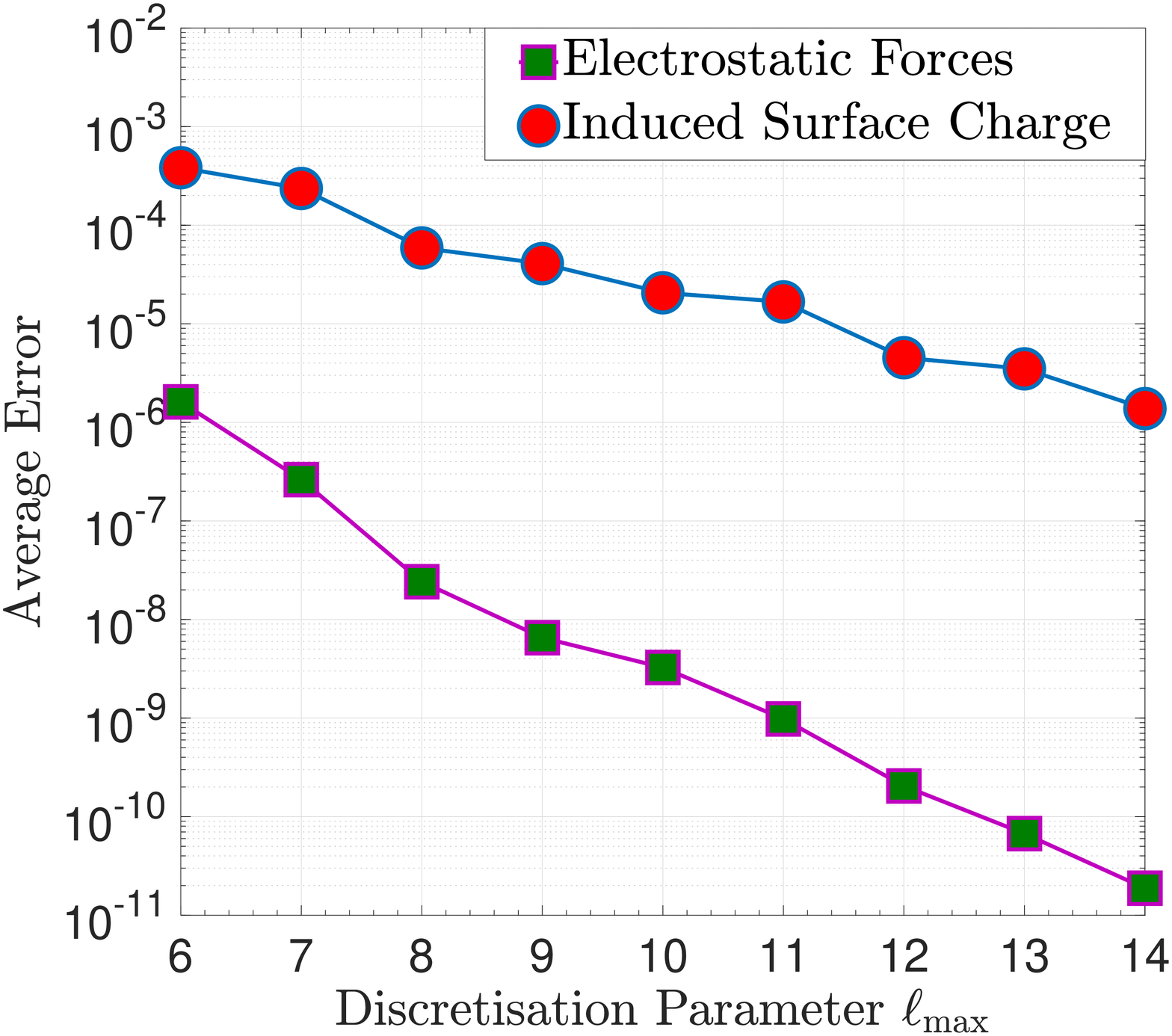} 
			\caption{Numerical results for the second type of lattice structure (Figure \ref{fig:12}).}
			\label{fig:22}
		\end{subfigure}
		\caption{Log-lin plots of the average error in the electrostatic forces and induced surface charge as a function of the discretisation parameter $\ell_{\max}$.  For comparison, the force on each particle is $\mathcal{O}(10^{-2})$ and the charge on each sphere is $\mathcal{O}(1)$.}
	\end{figure}

	\noindent {\textbf{Test 2: Dependence on the Separation Distance}}~
	We now wish to explore in more detail, the dependence of the error in the approximate forces on the minimal inter-sphere separation distance. We recall that the pre-factor $C_{\rm force}$ appearing in the error estimate \eqref{eq:rates_force} for the approximate forces (see Theorem \ref{thm:Force}) depends both on the coercivity constant $c_{\mathcal{V}}$ of the single layer boundary operator as well as on the pre-factor $C_{\rm charges}$ appearing in the error estimate for the induced surface charge (see Theorem \ref{thm:well-posed}). It was shown in the contribution \cite{Hassan1} (see also Lemma \ref{lem:single}) that the constants $c_{\mathcal{V}}$ and $C_{\rm charges}$ grow \emph{at most} as $\mathcal{O}\big(\frac{1}{{\delta}}\big)$ and $\mathcal{O}\big(\frac{1}{\sqrt{\delta}}\big)$ respectively for small $\delta$ where $\delta$ is the minimum inter-sphere separation distance. Consequently, we would expect the error in the approximate forces to also grow as the inter-sphere separation decreases.
	
	We consider two dielectric spheres placed on the $z$-axis at a separation of $s$ with identical dielectric constants $\kappa_1=\kappa_2=100$, fixed radius $r_1=1$ and varying radius $r_2$, and carrying unit negative and positive charge respectively. In order to obtain the true forces $\{\boldsymbol{F}_{i}\}_{i=1}^2$ for very small separations $s$, it is necessary to compute the reference solution $\nu$ to the BIE \eqref{eq:3.3a} using an extremely high value of the discretisation parameter $\ell_{\max}$. Indeed, our numerical tests indicate that an accurate approximation of the reference solution $\nu$ requires that $\ell_{\max} \approx \mathcal{O}(100)$. Our choice of geometry is thus deliberate since the axisymmetry allows us to consider an approximation space~$\widetilde{W}^{\ell_{\max}}\subset W^{\ell_{\max}}$ consisting of only axisymmetric local spherical harmonics expansions.

		\begin{figure}[h]
		\centering
		\begin{subfigure}[t]{0.48\textwidth}
			\centering
			\includegraphics[width=1\textwidth]{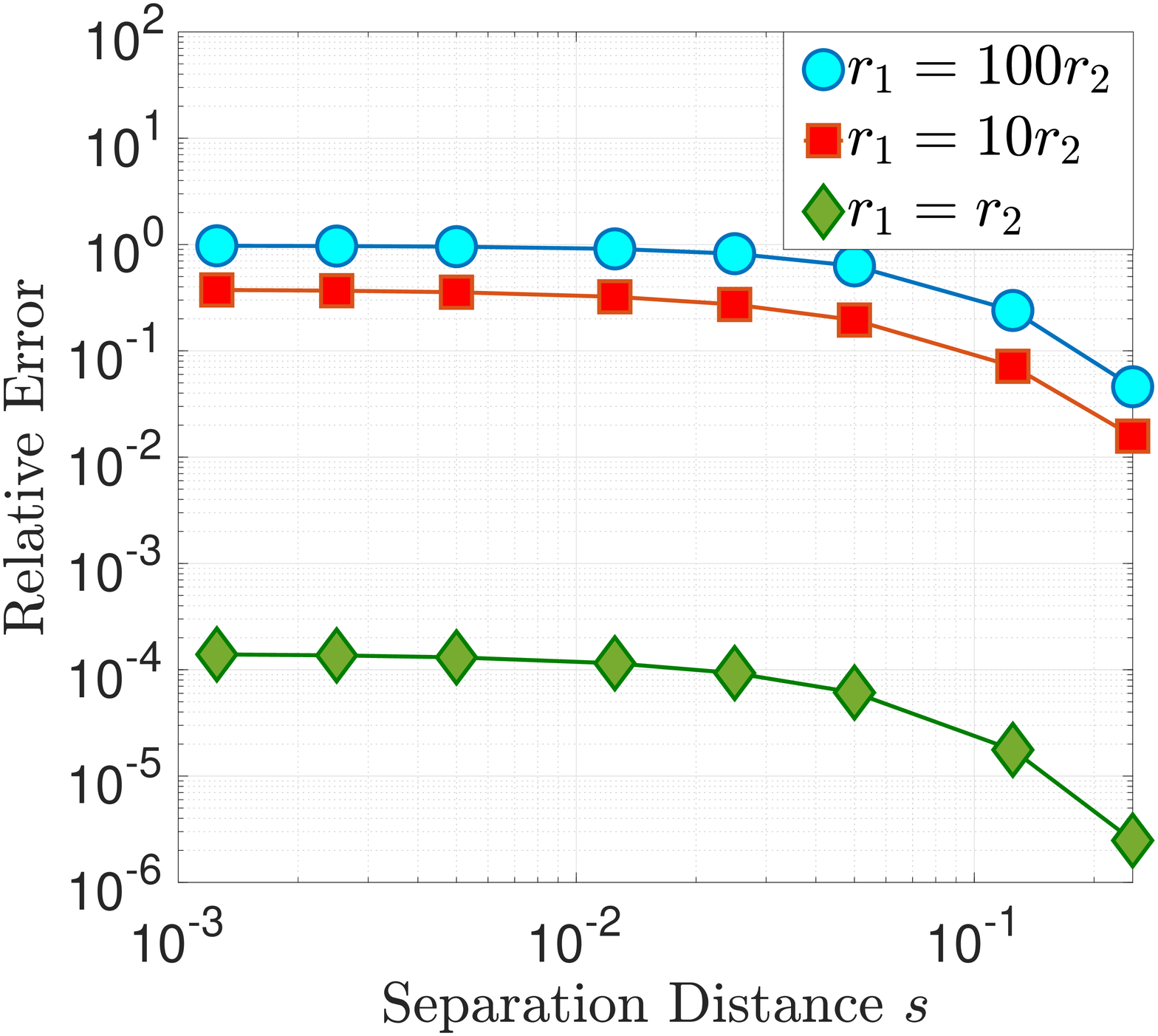} 
			\caption{Relative Error in the electrostatic forces as a function of the separation distance $s$.}
			\label{fig:31}
		\end{subfigure}\hfill
		\begin{subfigure}[t]{0.48\textwidth}
			\centering
			\includegraphics[width=1\textwidth]{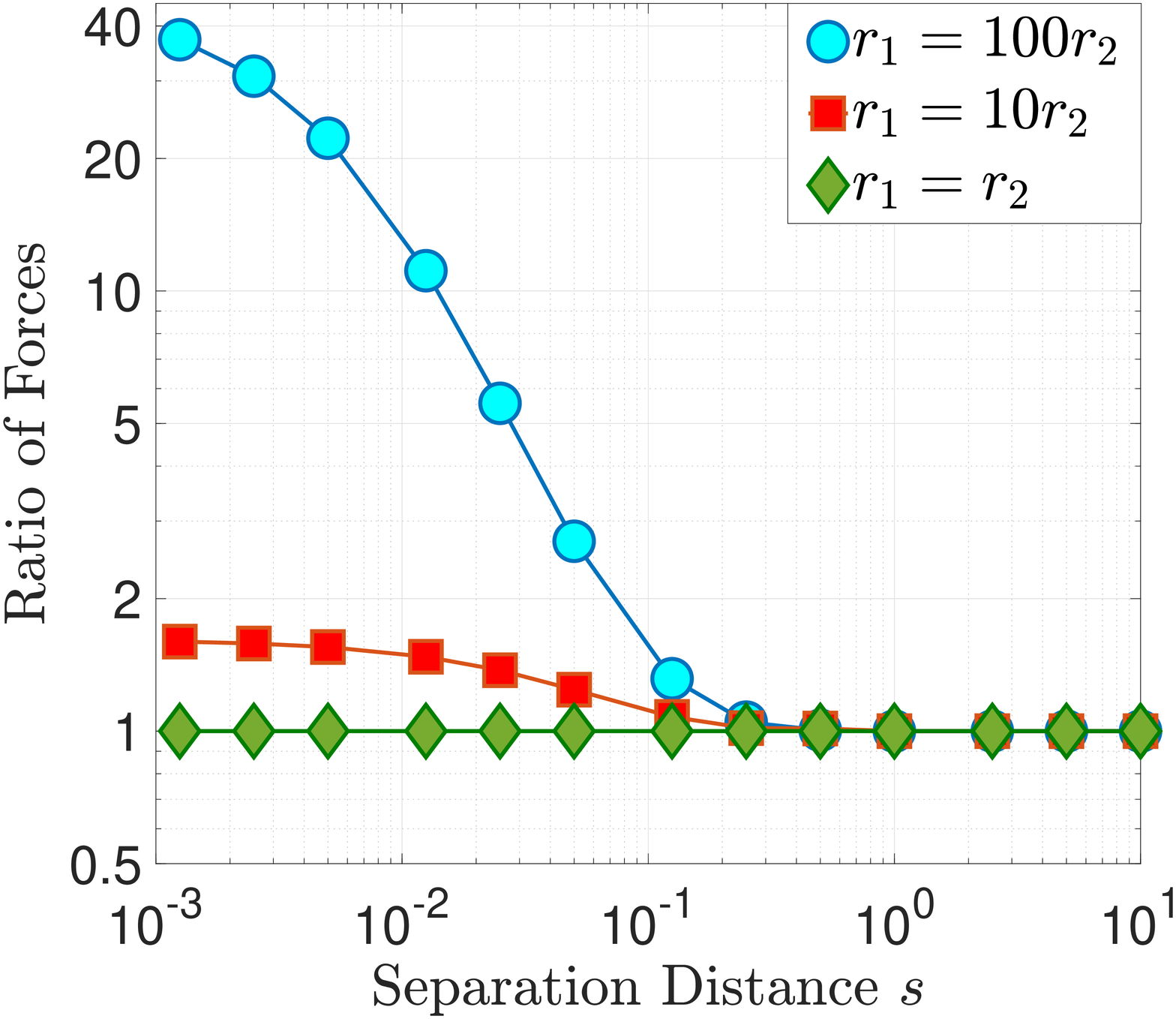} 
			\caption{Ratio of the exact and approximate forces on the first sphere as a function of the separation distance $s$.}
			\label{fig:32}
		\end{subfigure}
		\caption{The effects of the separation distance $s$ on the approximation errors of the electrostatic forces  for fixed discretisation parameter $\ell_{\max}=10$. }
	\end{figure}

	Figure \ref{fig:31} displays the relative error in the approximate electrostatic forces $\{\boldsymbol{F}^{\ell_{\max}}_{i}\}_{i=1}^2$ for $\ell_{\max}=10$. We immediately observe that if the radii of the two spheres are comparable, then the relative errors increase for decreasing $s$ but quickly reach a plateau that is much smaller than one. This indicates that while the relative error in the electrostatic forces does indeed grow for decreasing separation and a fixed $\ell_{\max}$, the forces are still being approximated with a certain degree of accuracy. In contrast, we see that if $r_2 \ll 1$, then the relative errors quickly approach one for small values of $s$, which indicates that the approximations of the forces for this setting essentially become worthless. This conclusion is supported by Figure \ref{fig:32} where we plot the ratio of the exact and approximate force on the first sphere. We observe that the ratio remains close to one if $r_2=r_1=1$  but explodes if $r_2=0.01$. This behaviour is explained by the fact that if $r_2 \ll 1$ and $s\to 0$, then the induced surface charge $\nu$ on the first sphere approaches a singularity at the point of contact, which is poorly represented in the approximation space $W^{\ell_{\max}}$.
	
	{ Let us remark here that a possible strategy for the treatment of point singularities that arise due to small separation distances between the particles has, for instance, been proposed in the contribution \cite{MR3493124}. The authors in \cite{MR3493124} derive analytical expressions for the induced potential both inside and outside a dielectric spherical particle due to a general multipole source using the method of image charges and image potentials. These analytical expressions are then combined with the classical method of moments to construct a hybrid algorithm. Numerical experiments indicate that the hybrid method has significantly better accuracy than the classical method of moments and also leads to solution matrices that do not suffer from ill-conditioning.}\\


	\noindent {\textbf{Test 3: $N$-independence of the Errors}}~
	Next, we demonstrate that the average error in the approximate forces $\{\boldsymbol{F}_i^{\ell_{\max}}\}_{i=1}^N$ is independent of the number $N$ of dielectric particles. We again consider the two types of lattices displayed in Figures \ref{fig:11} and \ref{fig:12}, and we increase $N$ simply by increasing the size of each lattice.
	
	Figures \ref{fig:41} and \ref{fig:42} display the average errors in the approximate electrostatic forces $\{\boldsymbol{F}_i^{\ell_{\max}}\}_{i=1^N}$ as a function of $N$ for three choices of the discretisation parameter, namely, $\ell_{\max}=6$, $\ell_{\max}=9$ and $\ell_{\max}=12$. As before, the true forces $\{\boldsymbol{F}_i\}_{i=1}^N$ were obtained from the reference solution $\nu$ to the BIE \eqref{eq:3.3a}, which was calculated by setting $\ell_{\max}=20$. Clearly the numerical results agree with the $N$-independent error estimate established by Theorem \ref{thm:Force}. We remark that since we are using the FMM with a single level octree, the computational cost of obtaining reference solutions scales as $\mathcal{O}\big(N^2\big)$ which limits the total number of spheres we consider to $N=2197$.
	
	\begin{figure}[h]
		\centering
		\begin{subfigure}[t]{0.48\textwidth}
			\centering
			\includegraphics[width=1\textwidth]{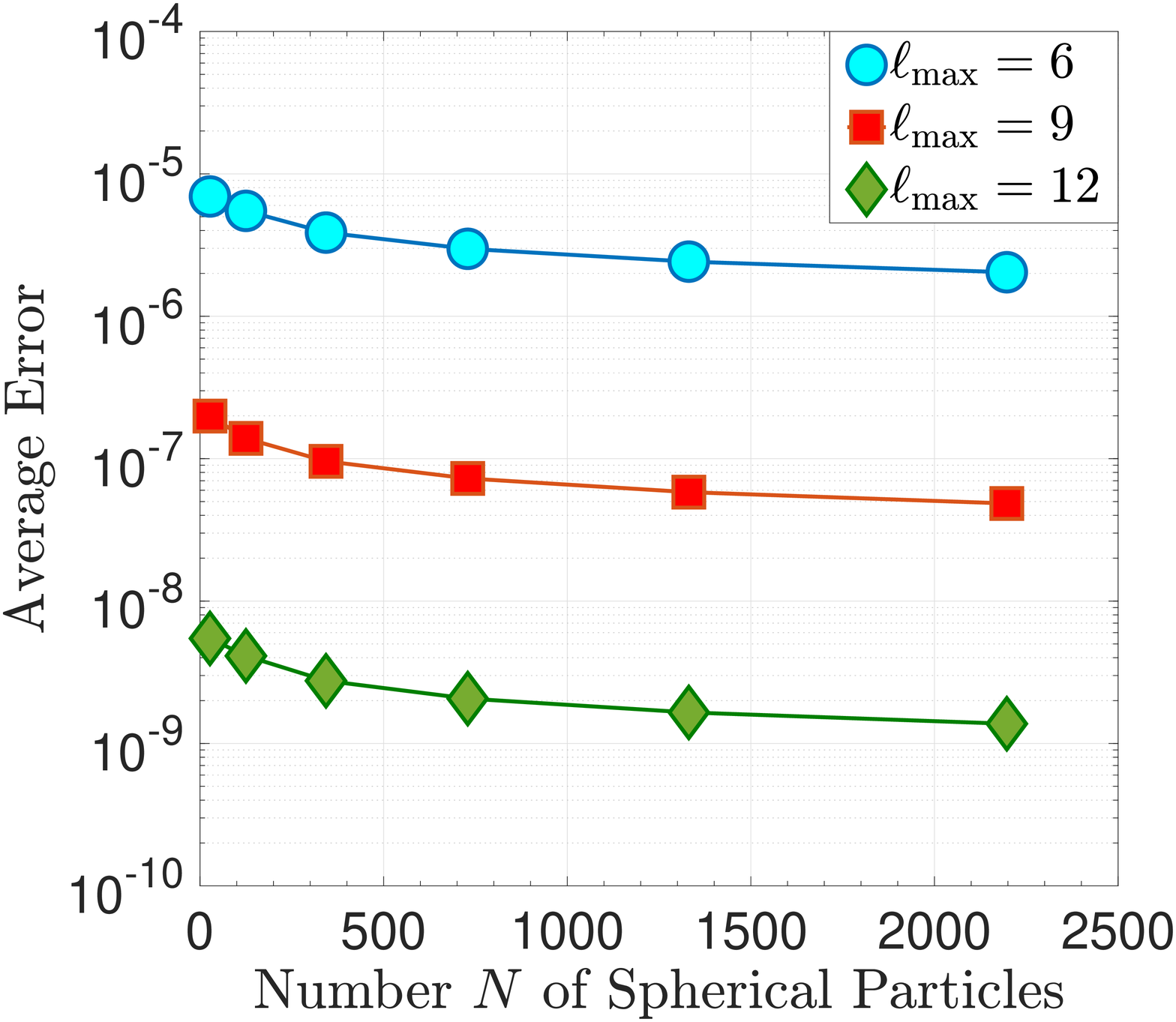} 
			\caption{Numerical results for the first type of lattice structure (Figure \ref{fig:11}). }
			\label{fig:41}
		\end{subfigure}\hfill
		\begin{subfigure}[t]{0.48\textwidth}
			\centering
			\includegraphics[width=1\textwidth]{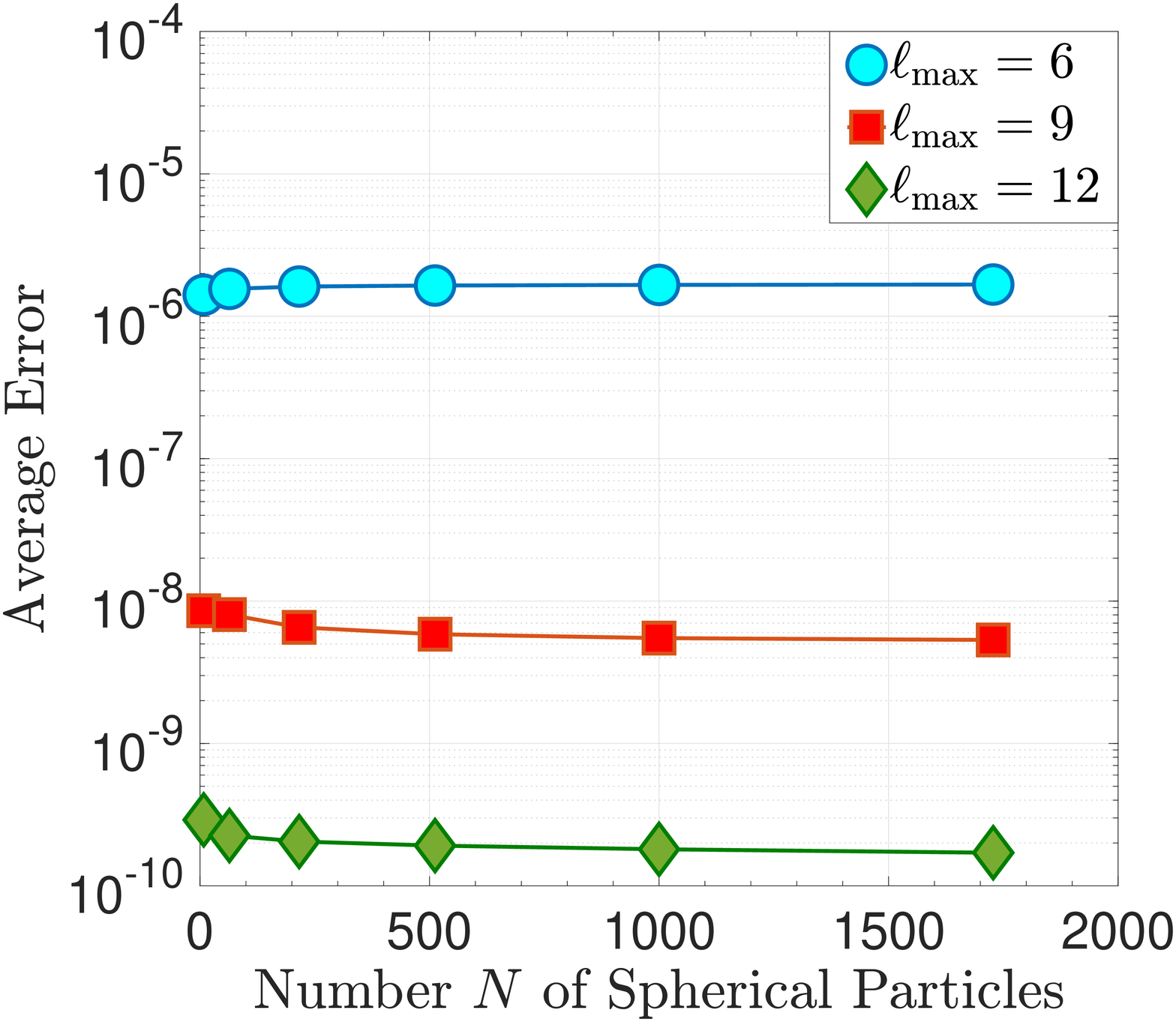} 
			\caption{Numerical results for the second type of lattice structure (Figure \ref{fig:12}).  }
			\label{fig:42}
		\end{subfigure}
		\caption{The average errors in the electrostatic forces as a function of the number $N$ of spherical dielectric particles. For comparison, the force on each particle is $\mathcal{O}(10^{-2})$.}
	\end{figure}
	
	\noindent {\textbf{Test 4: Linear Scaling Computation of the Forces}}~
	The goal of this final set of numerical experiments is to demonstrate that the approximate electrostatic forces $\{\boldsymbol{F}_i^{\ell_{\max}}\}$ can indeed be computed in $\mathcal{O}(N)$ operations for an increasing number $N$ of dielectric spherical particles. In order to achieve this linear scaling behaviour for a given $\ell_{\max}$ and increasing $N$, it is necessary to carefully adjust the two main FMM parameters, i.e., the number of levels $D$ in the octree structure of the bounding box containing all multipole sources, and the maximal degree $P$ of spherical harmonics used in the multipole expansion of the FMM kernel. We remark that the choice of $D$ depends only on the number $N$ of dielectric particles and the choice of $P$ depends only on the discretisation parameter $\ell_{\max}$.

	In the contribution \cite{Hassan2}, the authors performed a detailed numerical study to obtain appropriate values of $D$ and $P$ for dielectric particles arranged in lattice-like configurations. As a rough guide, it was proposed that
	\begin{itemize}
		\item $D$ should be picked so that there are between 4 and 32 particles in each leaf of the FMM octree with a preferred average of $8$. { Note that for an increasing number $N$ of particles, one must increase $D$ in order to achieve the linear complexity of the FMM. On the other hand, if $D$ is too large, then the FMM error could dominate the discretisation error leading to erroneous results (see \cite{Hassan2} for an in-depth discussion).}
		\item $P$ should be fixed so that $P\geq2\ell_{\max}$. Since the computational cost of each FMM call grows as $\mathcal{O}\big(P^3\big)$, it is preferable to pick $P$ as small as possible.
	\end{itemize}
	
	\begin{figure}[h]
		\centering
		\begin{subfigure}[t]{0.48\textwidth}
			\centering
			\includegraphics[width=\textwidth]{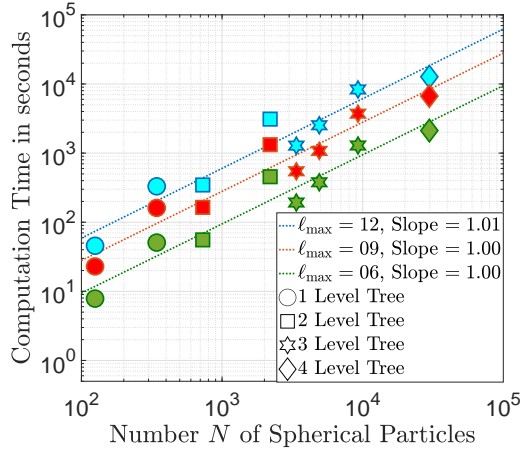} 
			\caption{Computation times for the first type of lattice structure (Figure \ref{fig:11}).}
			\label{fig:51}
		\end{subfigure}\hfill
		\begin{subfigure}[t]{0.48\textwidth}
			\centering
			\includegraphics[width=\textwidth]{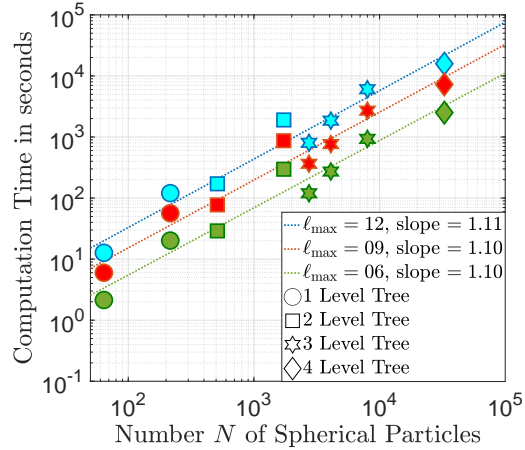} 
			\caption{Computation times for the second type of lattice structure (Figure \ref{fig:12}).}
			\label{fig:52}
		\end{subfigure}
		\caption{Computation times of the electrostatic forces as a function of the number $N$ of spherical dielectric particles.}
	\end{figure}

	Equipped with this methodology for picking the FMM parameters $P$ and $D$, we compute the approximate electrostatic forces $\{\boldsymbol{F}_i^{\ell_{\max}}\}_{i=1}^N$ for the two types of lattice structures \ref{fig:11} and \ref{fig:12} and the  three cases $\ell_{\max}=6, 9 $ and $12$. All numerical simulations were performed on a 2016 MacBook laptop with a 2.6 GHz Intel Core~i7 processor and 16GB of 2133 MHz LPDDR3 memory. Additionally, we set the linear solver tolerance to $10^{-6}$, $10^{-8}$ and $10^{-10}$ and the FMM parameter $P$ to $15, 20$ and $25$ in the cases $\ell_{\max}=6$, $\ell_{\max}=9$ and $\ell_{\max}=12$ respectively. Our results are displayed in Figures \ref{fig:51} and \ref{fig:52} and indicate excellent agreement with linear scaling behaviour.

	\section{Conclusion and Outlook}
	
	In this article, we have proposed and analysed an efficient numerical method for the computation of the electrostatic forces acting on a collection of dielectric spherical particles embedded in a homogenous polarisable medium and undergoing mutual polarisation. Our method is based on the Galerkin boundary integral equation framework proposed by Lindgren et al. \cite{lindgren2018} for the computation of the electrostatic energy of this system of dielectric particles, and uses the fast multipole method to compute matrix vector products involving the underlying solution matrix in linear scaling complexity. 
	
	{ Our main result is to prove that under appropriate assumptions on the types of geometrical configurations we consider, our proposed numerical method achieves \underline{linear} \underline{scaling in accuracy} for the computation of the forces, i.e., given a system composed of $N$ dielectric spherical particles, it requires only $\mathcal{O}(N)$ operations to calculate the approximate forces with a given average or relative error. In order to establish this result, we derived $N$-independent convergence rates for the approximate forces which yielded, as a corollary, exponential convergence of the approximate forces under suitable regularity assumptions. We also provided numerical evidence supporting our theoretical results.}

	There are two natural extensions of this work. First, the dielectric particles considered herein are assumed to have homogenous (but not necessarily identical) dielectric constants. A next step could be to extend our analysis and proposed method to spherical particles with spatially varying dielectric constants. Additionally, it would be of interest to explore if similar results can be proven in the case of forces arising from different potentials. In particular, one could consider the screened Coulomb potential encountered in the study of dielectric particles in an ionic solvent.


\bibliographystyle{siam}    
\bibliography{refs.bib}

\newpage
\appendix

\section{Electrostatic Energy-Based Definition of the Forces}\label{sec:appendix}

As discussed in Section \ref{sec:3} of our article, the electrostatic forces between charged dielectric particles can be defined either in terms of the electric field-- an approach found in the physics literature-- or in terms of the electrostatic energy, which is the approach favoured by the computational chemistry community. The goal of the following exposition is to present the electrostatic energy-based definition of the electrostatic forces and to demonstrate that this formalism is mathematically equivalent to the electric field-based definition presented in Section \ref{sec:3} of our article and used to derive error estimates in Section~\ref{sec:4}.

In the sequel, we will assume the setting of Sections \ref{sec:2} and \ref{sec:3} of our article. We begin with the formal definition of the energy function and the exact electrostatic energy. 

\begin{definition}[Energy Function and Exact Electrostatic Energy]\label{def:Energy}~
	We define the energy function $\mathcal{E} \colon H^{-\frac{1}{2}}(\partial \Omega) \times H^{-\frac{1}{2}}(\partial \Omega) \rightarrow $ as the mapping with the property that for all $\sigma_1, \sigma_2 \in H^{-\frac{1}{2}}(\partial \Omega)$ it holds that
	\begin{align*}
	\mathcal{E}(\sigma_1, \sigma_2):=  \frac{1}{2}\,4\pi\,\langle  \sigma_1, \mathcal{V}\sigma_2\rangle_{\partial \Omega}=\frac{1}{2}\,4\pi\,\langle  \sigma_2, \mathcal{V}\sigma_1\rangle_{\partial \Omega}.
	\end{align*}
	
	Furthermore, if $\sigma_f \in H^{-\frac{1}{2}}(\partial \Omega)$ and $\nu\in H^{-\frac{1}{2}}(\partial \Omega)$ denotes the solution to the boundary integral equation \eqref{eq:3.3a} with right-hand side generated by $\sigma_f$, then we define the exact electrostatic energy of the system of $N$ dielectric particles carrying free charge $\sigma_f$ as 
	\begin{align*}
	\mathcal{E}_{\sigma_f}^{\rm exact}:=\mathcal{E}(\sigma_f, \nu).
	\end{align*}
\end{definition}

\begin{remark}
	Consider Definition \ref{def:Energy} of the total electrostatic energy. The factor $4\pi$ appears in the this definition as a pre-factor because the right-hand side of the BIE \eqref{eq:3.3a} contains the term $4\pi$. Obviously, this factor has no bearing on the analysis.
\end{remark}

The exact electrostatic forces are now defined as follows.

\begin{definition}[Second Definition of the Forces]\label{def:Force1}~
	Let $\sigma_f \in H^{-\frac{1}{2}}(\partial \Omega)$ be a given free charge, let $\nu \in H^{-\frac{1}{2}}(\partial \Omega)$ denote the unique solution to the BIE \eqref{eq:3.3a} with right-hand side generated by $\sigma_f$, and let $\mathcal{E}_{\sigma_f}^{\rm exact}$ denote the total electrostatic energy of this system as defined by Definition \ref{def:Energy}. Then for each $i=1, \ldots, N$ we define the net force acting on the dielectric particle represented by $\Omega_i$ as the vector $\widetilde{\boldsymbol{F}}_i \in \mathbb{R}^3$ given by
	\begin{align*}
	\widetilde{\boldsymbol{F}}_i:= -\nabla_{\bold{x}_i} \mathcal{E}_{\sigma_f}^{\rm exact},
	\end{align*}
	where the gradient is taken with respect to the location $\bold{x}_i \in \mathbb{R}^3$ of the centre of the open ball $\Omega_i$.
\end{definition}

Some remarks are now in order.

\begin{remark}
	Consider Definitions \ref{def:Energy} and \ref{def:Force1}. We observe that the total electrostatic energy $\mathcal{E}_{\sigma_f}^{\rm exact}$ is, in particular, a function of the induced surface charge $\nu$, and since $\nu$ is the solution to the BIE \eqref{eq:3.3a}, it will implicitly depend on the locations $\{\bold{x}_i\}_{i=1}^N$ of the centres of the open balls $\{\Omega_i\}_{i=1}^N$. Thus, the exact electrostatic energy also implicitly depends on $\{\bold{x}_i\}_{i=1}^N$. 
\end{remark}

\begin{remark}
	It is possible to give an intuitive interpretation of Definition \ref{def:Force1} of the electrostatic forces. Indeed, assume that the free charge $\sigma_f$ and other physical parameters such as the dielectric constants and the radii $\{r_i\}_{i=1}^N$ of the open balls $\{\Omega_i\}_{i=1}^N$ are fixed. Then the resulting induced surface charge $\nu$ is uniquely determined by the locations $\{\bold{x}_i\}_{i=1}^N$ of the centres of the open balls $\{\Omega_i\}_{i=1}^N$. Thus the electrostatic energy $\mathcal{E}(\sigma_f, \cdot)$ can be viewed as a function of $\{\bold{x}_i\}_{i=1}^N$. The set of values of $\mathcal{E}(\sigma_f, \cdot)$ for all admissible sphere centres $\{\bold{x}_i\}_{i=1}^N$ defines a so-called potential energy surface (PES), and the graph of this PES is a $3N$-dimensional manifold. Consequently, given a fixed choice of sphere centres $\{\bold{x}_i\}_{i=1}^N$, the force acting on each dielectric particle (up to a scaling factor) is given by the negative gradient of the PES at the point $\{\bold{x}_i\}_{i=1}^N$.
\end{remark}

Definition \ref{def:Force1} of the electrostatic forces requires us to first compute the exact electrostatic energy $\mathcal{E}_{\sigma_f}^{\rm exact}=\mathcal{E}(\sigma_f, \nu)$. Of course in practice, $\mathcal{E}^{\rm exact}_{\sigma_f}$ is not known since $\sigma_f\in H^{-\frac{1}{2}}(\partial \Omega)$ may be infinite-dimensional and the exact induced surface charge $\nu \in H^{-\frac{1}{2}}(\partial \Omega)$ that solves the BIE \eqref{eq:3.3a} is not known. It is therefore necessary to define \emph{approximate} electrostatic forces in terms of a \emph{discrete} electrostatic energy. More precisely, we have the following definitions:

\begin{definition}[Discrete Electrostatic Energy]\label{def:Energy_approx}~
	Let $\sigma_f \in H^{-\frac{1}{2}}(\partial \Omega)$ be a given free charge, let $\ell_{\max} \in \mathbb{N}$, and let $\nu_{\ell_{\max}}$ be the unique solution to the Galerkin discretisation (2.4) with right-hand side generated by $\sigma_f$. We define the discrete electrostatic energy of the system of $N$ dielectric particles carrying free charge $\sigma_f$ as
	\begin{align*}
	\mathcal{E}_{\sigma_f}^{\ell_{\max}}:= \mathcal{E}(\mathbb{Q}_{\ell_{\max}}\sigma_f, \nu_{\ell_{\max}}).
	\end{align*}
	
\end{definition}

The approximate forces are then defined as follows.

\begin{definition}[Second Definition of the Approximate Forces]\label{def:Force1_approx}~
	Let $\sigma_f \in H^{-\frac{1}{2}}(\partial \Omega)$ be a given free charge, let $\ell_{\max} \in \mathbb{N}$, let $\nu_{\ell_{\max}}$ be the unique solution to the Galerkin discretisation \eqref{eq:Galerkina} with right-hand side generated by $\sigma_f$, and let $\mathcal{E}_{\sigma_f}^{\ell_{\max}}$ denote the discrete electrostatic energy of this system as defined by Definition~\ref{def:Energy_approx}. Then for each $i=1, \ldots, N$ we define the approximate net force acting on the dielectric particle represented by $\Omega_i$ as the vector $\widetilde{\boldsymbol{F}}^{\ell_{\max}}_i \in \mathbb{R}^3$ given by
	\begin{align*}
	\widetilde{\boldsymbol{F}}^{\ell_{\max}}_i:= -\nabla_{\bold{x}_i} \mathcal{E}_{\sigma_f}^{\ell_{\max}},
	\end{align*}
	where the gradient is taken with respect to the location $\bold{x}_i \in \mathbb{R}^3$ of the centre of the open ball $\Omega_i$.
\end{definition}

\begin{remark}
	Consider Definition \ref{def:Energy_approx} of the discrete electrostatic energy. In analogy with the exact electrostatic energy case, we observe that the discrete electrostatic energy  $\mathcal{E}(\mathbb{Q}_{\ell_{\max}}\sigma_f, \cdot)$ defines a discrete potential energy surface (dPES) for different locations of the sphere renters $\{\bold{x}_i\}_{i=1}^N$. Moreover, considering Definition \ref{def:Force1_approx}, we see that the approximate electrostatic force is defined precisely in terms of the negative gradient of the dPES at the point $\{\bold{x}_i\}_{i=1}^N$.
	
	It can now be seen why the computational chemistry community finds Definition~\ref{def:Force1} of the electrostatic forces appealing. Indeed, suppose that we wish to numerically simulate the movement of charged dielectric particles due to the electrostatic forces. Then at each given time step, we have by construction that the approximate electrostatic forces are consistent with the dPES. Consequently, if one uses a symplectic method to perform time-integration, then the total (discrete) energy of the system can be maintained as a conserved quantity (up to a perturbation).
\end{remark}

Finally, we have the following result on the equivalence between the electrostatic energy-based formalism and electric field-based methodology for defining the electrostatic forces. 
\begin{theorem}\label{thm:equivalence}
	Let $\sigma_f \in H^{-\frac{1}{2}}(\partial \Omega)$ be a given free charge, let ${\ell_{\max}} \in \mathbb{N}$, let $\nu \in H^{-\frac{1}{2}}(\partial \Omega)$ and $\nu_{\ell_{\max}} \in W^{\ell_{\max}}$ denote the solutions to the BIE \eqref{eq:3.3a} and Galerkin discretisation \eqref{eq:Galerkina} respectively with right-hand sides generated by $\sigma_f$, and for each $i \in \{1, \ldots, N\}$ let $\boldsymbol{F}_i, \widetilde{\boldsymbol{F}}_i \in \mathbb{R}^3$ denote the exact electrostatic forces as defined by Definitions \ref{def:Force2} and \ref{def:Force1} respectively and let $\boldsymbol{F}^{\ell_{\max}}_i, \widetilde{\boldsymbol{F}}^{\ell_{\max}}_i \in \mathbb{R}^3 $ denote the approximate electrostatic forces as defined by Definitions \ref{def:Force2_approx} and \ref{def:Force1_approx} respectively. Then for all $i \in \{1, \ldots, N\}$ it holds that
	\begin{align*}
	\widetilde{\boldsymbol{F}}_i = \boldsymbol{F}_i, \qquad \text{and} \qquad \widetilde{\boldsymbol{F}}^{\ell_{\max}}_i=\boldsymbol{F}^{\ell_{\max}}_i.
	\end{align*}
\end{theorem}

{	We will prove Theorem \ref{thm:equivalence} for the approximate forces. The proof for the exact forces is similar in spirit with an additional complication due to the fact that the exact induced surface charge $\nu \in H^{-\frac{1}{2}}(\partial \Omega)$ is a distribution. Consequently, extra care must be taken when performing direct calculations involving the explicit, integral representation of the single layer boundary operator $\mathcal{V}\colon H^{-\frac{1}{2}}(\partial \Omega)\rightarrow H^{\frac{1}{2}}(\partial \Omega)$.

	To facilitate the proof of Theorem \ref{thm:equivalence} in the case of the approximate forces, it is advantageous to represent elements of the approximation space $W^{\ell_{\max}}$ as vectors in Euclidean space. This requires the introduction of a basis on $W^{\ell_{\max}}$, and in view of Definition \ref{def:6.7}, the natural choice of basis functions are the local spherical harmonics on each sphere. 
	
	\begin{definition}[Choice of Basis]\label{def:Basis}~
		Let $\ell_{\max} \in \mathbb{N}$. For each $j \in \{1, \ldots, N\}$ and all $\ell \in \{0, \ldots, \ell_{\max}\}$, $-\ell \leq m \leq \ell$ we define the function ${\mathcal{Y}}^j_{\ell m} \colon \partial \Omega \rightarrow \mathbb{R}$ as
		\begin{align*}
		{\mathcal{Y}}^j_{\ell m}(\bold{x}):= \begin{cases}\mathcal{Y}_{\ell}^{m} \left(\frac{\bold{x}-\bold{x}_j}{\vert \bold{x}-\bold{x}_j\vert}\right) \quad &\text{for all } \bold{x} \in \partial \Omega_j,\\
		0 \quad &\text{otherwise},
		\end{cases}
		\end{align*}
		and we equip the approximation space $W_{\ell_{\max}}$ with the basis $\{\mathcal{Y}^j_{\ell m}\}$.
	\end{definition}
	
	\noindent 	\textbf{Notation:} Let $\ell_{\max} \in \mathbb{N}$. We will henceforth denote by $M:= N \cdot (\ell_{\max}+1)^2$, the dimension of the approximation space $W^{\ell_{\max}}$. 
	
	\begin{remark}\label{rem:vectors}
		Consider Definition \ref{def:Basis} of the basis functions on $W^{\ell_{\max}}$. These functions establish an isomorphism between $W^{\ell_{\max}}$ and $\mathbb{R}^{M}$. Indeed, we associate an arbitrary $\psi \in W^{\ell_{\max}}$ with $\boldsymbol{\psi} \in \mathbb{R}^{M}$ defined as
		\begin{align*}
		[\boldsymbol{\psi}_{i}]_{\ell}^m:= \left(\psi, \mathcal{Y}^i_{\ell m}\right)_{L^2(\partial \Omega_i)}, \text{ for } i\in \{1, \ldots, N\},~ \ell \in \{0, \ldots, \ell_{\max}\} ~\text{and } -\ell \leq m \leq \ell.
		\end{align*}
		Consequently, given functions in the space $W^{\ell_{\max}}$, we will often refer to their vector representations in $\mathbb{R}^{M}$ and vice versa. Moreover, to facilitate identification we will frequently use bold symbols for the vector representations.
	\end{remark}

	In order to present a concise proof of Theorem \ref{thm:equivalence}, it is useful to define the boundary integral operators associated with the BIE \eqref{eq:3.3a} and Galerkin discretisation \eqref{eq:Galerkina}.
	\begin{definition}\label{def:A}
		We define the linear operator $\mathcal{A} \colon {H}^{\frac{1}{2}}(\partial \Omega) \rightarrow {H}^{\frac{1}{2}}(\partial \Omega)$ as the mapping with the property that for all $\lambda \in {H}^{\frac{1}{2}}(\partial \Omega)$ it holds that
		\begin{align*}
		\mathcal{A} \lambda:= \lambda - \mathcal{V} \left(\frac{\kappa_0-\kappa}{\kappa_0} \text{\rm DtN}\lambda\right).
		\end{align*}
		
		In addition, we define $\mathcal{A}^* \colon {H}^{-\frac{1}{2}}(\partial \Omega) \rightarrow {H}^{-\frac{1}{2}}(\partial \Omega)$ as the adjoint operator of $\mathcal{A}$.
	\end{definition}

	\begin{definition}\label{def:Matrices}
		Let $\ell_{\max} \in \mathbb{N}_0$, let $\sigma_f \in H^{-\frac{1}{2}}(\partial \Omega)$ be a given free charge, let $\mathcal{V} \colon H^{-\frac{1}{2}}(\partial \Omega) \rightarrow H^{\frac{1}{2}}(\partial \Omega)$ and $\textbf{\rm DtN}  \colon H^{\frac{1}{2}}(\partial \Omega) \rightarrow H^{-\frac{1}{2}}(\partial \Omega) $ denote the single layer boundary operator and Dirichlet-to-Neumann map respectively, let $\mathcal{A}\colon H^{\frac{1}{2}}(\partial \Omega) \rightarrow H^{\frac{1}{2}}(\partial \Omega)$ denote the boundary integral operator defined through Definition \ref{def:A}, and let $\mathbb{Q}_{\ell_{\max}} \colon H^{-\frac{1}{2}}(\partial \Omega) \rightarrow W^{\ell_{\max}}$ and $\mathbb{P}_{\ell_{\max}} \colon H^{\frac{1}{2}}(\partial \Omega) \rightarrow W^{\ell_{\max}}$ denote the projection operators defined through Definition \ref{def:PQ}. Then 
		\begin{itemize}
			\item We define the vector $\boldsymbol{\sigma_f} \in \mathbb{R}^{M}$ as
			\begin{align*}
			[\boldsymbol{\sigma_f}_{i}]_{\ell}^m:= \left(\mathbb{Q}_{\ell_{\max}}\sigma_f, \mathcal{Y}^i_{\ell m}\right)_{L^2(\partial \Omega_i)},
			\end{align*}
			where  $i\in \{1, \ldots, N\}, ~ \ell \in \{0, \ldots, \ell_{\max}\} ~ \text{and } ~|m| \leq \ell$. \vspace{5mm}
			
			\item We define the diagonal matrix $\boldsymbol{\rm DtN}^{\kappa} \in \mathbb{R}^{M\times M}$ as 
			\begin{align*}
			[\boldsymbol{\rm DtN}^{\kappa}_{ij}]_{\ell \ell'}^{m m'}:= \delta_{ij}\left( \frac{\kappa_j-\kappa_0}{\kappa_0}\text{\rm DtN}\mathcal{Y}^j_{\ell'm'}, \mathcal{Y}^i_{\ell m}\right)_{L^2(\partial \Omega_i)},
			\end{align*}
			where  $ i, j\in \{1, \ldots, N\}, ~ \ell, \ell' \in \{0, \ldots, \ell_{\max}\} ~ \text{ and } ~|m| \leq \ell, ~ |m'| \leq \ell'$. \vspace{5mm}
			\item We define the symmetric, positive definite matrix $\boldsymbol{V} \in \mathbb{R}^{M\times M}$ as
			\begin{align*}
			[\boldsymbol{V}_{ij}]_{\ell \ell'}^{m m'}:= \left({\mathcal{V}}\mathcal{Y}^j_{\ell'm'}, \mathcal{Y}^i_{\ell m}\right)_{L^2(\partial \Omega_i)},
			\end{align*}
			where $i, j\in \{1, \ldots, N\}, ~ \ell, \ell' \in \{0, \ldots, \ell_{\max}\} ~ \text{ and } ~|m| \leq \ell, ~ |m'| \leq \ell'$.\vspace{5mm}
			
			\item We define the solution matrix $\boldsymbol{A} \in \mathbb{R}^{M\times M}$ as
			\begin{align*}
			[\boldsymbol{A}_{ij}]_{\ell \ell'}^{m m'}:= \left({\mathcal{A}}\mathcal{Y}^j_{\ell'm'}, \mathcal{Y}^i_{\ell m}\right)_{L^2(\partial \Omega_i)},
			\end{align*}
			where $i, j\in \{1, \ldots, N\}, ~ \ell, \ell' \in \{0, \ldots, \ell_{\max}\} ~ \text{ and } ~|m| \leq \ell, ~ |m'| \leq \ell'$.
		\end{itemize}
	\end{definition}
	
	Equipped with the matrix representations of the relevant boundary integrals, we are now ready to state the proof of Theorem \ref{thm:equivalence} for the approximate forces. 
	
	\begin{proof}[Proof of Theorem \ref{thm:equivalence}:]
		
		We assume the setting of Remark \ref{rem:vectors} and Definition~\ref{def:Matrices} and we denote by $\boldsymbol{\nu_{\ell_{\max}}} \in \mathbb{R}^M$ the vector representation of the solution $\nu_{\ell_{\max}} \in W^{\ell_{\max}}$ to the Galerkin discretisation \eqref{eq:Galerkina}. We divide the proof into two steps:
		\begin{itemize}
			\item We first show that for each $i \in \{1, \ldots, N\}$ and $\alpha =1,2,3$ it holds that \[\left(\widetilde{\boldsymbol{F}}^{\ell_{\max}}_i\right)_{\alpha}=-\frac{1}{2}\kappa_0\, \boldsymbol{\nu_{\ell_{\max}}} \cdot \left(\partial_{\bold{x}_i^{\alpha}}\boldsymbol{V}\right) \boldsymbol{\nu_{\ell_{\max}}},\] 
			where $\left(\widetilde{\boldsymbol{F}}^{\ell_{\max}}_i\right)_{\alpha}$ denotes the $\alpha^{\rm th}$ component of the approximate force $\widetilde{\boldsymbol{F}}^{\ell_{\max}}_i$ and $\partial_{\bold{x}_i^{\alpha}}$ denotes the $\alpha^{\rm th}$ component of the sphere-centred gradient $\nabla_{\bold{x}_i}$.
			\item In the second step we use this expression to show that $\widetilde{\boldsymbol{F}}^{\ell_{\max}}_i=\boldsymbol{F}^{\ell_{\max}}_i$.
		\end{itemize}
		\vspace{5mm}
		\textbf{Step 1:} Consider Definition \ref{def:Energy_approx} of the discrete electrostatic energy. A direct calculation shows that
		\begin{align*}
		\mathcal{E}_{\sigma_f}^{\ell_{\max}} = \frac{1}{2} 4 \pi \,\boldsymbol{\sigma_f} \cdot \boldsymbol{V} \boldsymbol{\nu_{\ell_{\max}}}.
		\end{align*}

		Let $i \in \{1, \ldots, N\}$ and $\alpha \in \{1, 2,3\}$ be fixed. Using Definition \ref{def:Force1_approx} of the approximate force and the fact that the vector $\boldsymbol{\sigma_f}$ is independent of the sphere centre locations $\left\{\bold{x}_i\right\}_{i=1}^N$ (see Definition \ref{def:Matrices}), we see that
		\begin{align}\label{eq:dual01}
		\left(\widetilde{\boldsymbol{F}}^{\ell_{\max}}_i \right)_{\alpha}=-\frac{1}{2} 4 \pi \,\boldsymbol{\sigma_f} \cdot \partial_{\bold{x}_i^{\alpha}}\boldsymbol{V} \boldsymbol{\nu_{\ell_{\max}}}=-\frac{1}{2} 4 \pi \,\boldsymbol{\sigma_f} \cdot \partial_{\bold{x}_i^{\alpha}}\boldsymbol{\lambda_{\ell_{\max}}},
		\end{align}
		where $\boldsymbol{\lambda_{\ell_{\max}}}:= \boldsymbol{V}\boldsymbol{\nu_{\ell_{\max}}} \in \mathbb{R}^{M}$ is the vector representation of $\lambda_{\ell_{\max}}:=\mathbb{P}_{\ell_{\max}}\mathcal{V}\nu_{\ell_{\max}} \in W^{\ell_{\max}}$, i.e., the so-called approximate surface electrostatic potential. Consequently, it suffices to compute the sphere-centred partial derivatives of $\boldsymbol{\lambda_{\ell_{\max}}}$. Using the Galerkin discretisation~\eqref{eq:Galerkina}, one can show (see~\cite{Hassan1}) that the vector $\boldsymbol{\lambda_{\ell_{\max}}}$ solves the finite-dimensional BIE
		\begin{align}\label{eq:dual11}
		\boldsymbol{A}\boldsymbol{\lambda_{\ell_{\max}}}= \frac{4\pi}{\kappa_0}\boldsymbol{V}\boldsymbol{\sigma_f}.
		\end{align}
		
		Consequently, taking the derivative on both sides of Equation \eqref{eq:dual11} and using the chain rule yields
		\begin{align*}
		\boldsymbol{A} \big( \partial_{\bold{x}_i^{\alpha}}\boldsymbol{\lambda_{\ell_{\max}}}\big)=  \frac{4\pi}{\kappa_0}\big( \partial_{\bold{x}_i^{\alpha}}\boldsymbol{V}\big)\boldsymbol{\sigma_f}-\big(\partial_{\bold{x}_i^{\alpha}}\boldsymbol{A}\big) \boldsymbol{\lambda_{\ell_{\max}}}.
		\end{align*}
		
		Next, using Definition \ref{def:A} of the boundary integral operator $\mathcal{A}$ we write the solution matrix $\boldsymbol{A}$ as
		\begin{align*}
		\boldsymbol{A}= \boldsymbol{I_{\ell_{\max}}} + \boldsymbol{V}\boldsymbol{\rm DtN}^{\kappa},
		\end{align*}
		where $\boldsymbol{I_{\ell_{\max}}}  \in \mathbb{R}^{M \times M}$ is the identity matrix. We now observe that both $\boldsymbol{I_{\ell_{\max}}}$ and the matrix $\boldsymbol{\rm DtN}^{\kappa}$ (see Definition \ref{def:Matrices}) are also independent of the sphere centre locations $\left\{\bold{x}_i\right\}_{i=1}^N$. Consequently, it holds that
		\begin{align*}
		-\partial_{\bold{x}_i^{\alpha}}\boldsymbol{A}= -\big(\partial_{\bold{x}_i^{\alpha}}\boldsymbol{V}\big)\boldsymbol{\rm DtN}^{\kappa}.
		\end{align*}
		
		A simple calculation then yields
		\begin{align}\label{eq:dual21}
		\boldsymbol{A} \big( \partial_{\bold{x}_i^{\alpha}}\boldsymbol{\lambda_{\ell_{\max}}}\big)= \big( \partial_{\bold{x}_i^{\alpha}}\boldsymbol{V}\big)\frac{4\pi}{\kappa_0}\boldsymbol{\sigma_f}-\big(\partial_{\bold{x}_i^{\alpha}}\boldsymbol{V}\big)\boldsymbol{\rm DtN}^{\kappa}\boldsymbol{\lambda_{\ell_{\max}}}=\big( \partial_{\bold{x}_i^{\alpha}}\boldsymbol{V}\big) \boldsymbol{\nu_{\ell_{\max}}},
		\end{align}
		where the last equality follows from the fact that $\boldsymbol{\nu_{\ell_{\max}}}=\frac{4\pi}{\kappa_0}\boldsymbol{\sigma_f}-\boldsymbol{\rm DtN}^{\kappa}\boldsymbol{\lambda_{\ell_{\max}}}$, which can be deduced directly from the Galerkin discretisation~\eqref{eq:Galerkina}.
		
		Next, let $\boldsymbol{A}^{\rm T} \in \mathbb{R}^{M \times M}$ denote the transpose of $\boldsymbol{A}$. Clearly, $\boldsymbol{A}^{\rm T} $ is the matrix representation (with respect to the basis \ref{def:Basis}) of the finite-dimensional operator $\mathbb{Q}_{\ell_{\max}}\mathcal{A}^* \mathbb{Q}_{\ell_{\max}} \colon W^{\ell_{\max}} \rightarrow W^{\ell_{\max}}$, i.e., the operator associated with the Galerkin discretisation~\eqref{eq:Galerkina}. Since the Galerkin discretisation~\eqref{eq:Galerkina} is well-posed, the matrices $\boldsymbol{A}^{\rm T} $ and $\boldsymbol{A}$ are both invertible. Consequently, we can use Equations \eqref{eq:dual01} and \eqref{eq:dual21} to write the approximate electrostatic force~$\widetilde{\boldsymbol{F}}^{\ell_{\max}}_i$ as
		\begin{align*}
		\left(\widetilde{\boldsymbol{F}}^{\ell_{\max}}_i \right)_{\alpha}=&-\frac{1}{2} 4 \pi \,\boldsymbol{\sigma_f} \cdot \partial_{\bold{x}_i^{\alpha}}\boldsymbol{\lambda_{\ell_{\max}}}=-\frac{1}{2} 4 \pi \,\boldsymbol{\sigma_f} \cdot \left(\boldsymbol{A}^{-1}\big( \partial_{\bold{x}_i^{\alpha}}\boldsymbol{V}\big) \boldsymbol{\nu_{\ell_{\max}}}\right)\\
		&=-\frac{1}{2} 4 \pi \,\big(\boldsymbol{A}^{\rm T}\big)^{-1}\boldsymbol{\sigma_f} \cdot \big( \partial_{\bold{x}_i^{\alpha}}\boldsymbol{V}\big) \boldsymbol{\nu_{\ell_{\max}}}
		\end{align*}
		
		Finally, in view of the Galerkin discretisation (2.4) we obtain that $\big(\boldsymbol{A}^{\rm T}\big)^{-1}\boldsymbol{\sigma_f}$ $= \frac{\kappa_0}{4\pi} \boldsymbol{\nu_{\ell_{\max}}}$ so that
		\begin{align}\label{eq:dual12}
		\left(\widetilde{\boldsymbol{F}}^{\ell_{\max}}_i \right)_{\alpha}=-\frac{1}{2} \kappa_0 \, \boldsymbol{\nu_{\ell_{\max}}}\cdot \big( \partial_{\bold{x}_i^{\alpha}}\boldsymbol{V}\big) \boldsymbol{\nu_{\ell_{\max}}}.
		\end{align}
		
		\textbf{Step 2:} We will now attempt to simplify the expression \eqref{eq:dual12} for the approximate electrostatic forces. To this end, let $\mathbb{S}^2\subset \mathbb{R}^3$ denote the unit sphere and for each $i \in \{1, \ldots, N\}$ let $\nu^{\ell_{\max}}_i:= \nu_{\ell_{\max}}\vert_{\partial \Omega_i}$. Using Definition \ref{def:Matrices} of the matrix $\boldsymbol{V}$ and a simple change of variables, Equation \eqref{eq:dual12} can be written in the form
		\begin{multline*}
		\widetilde{\boldsymbol{F}}^{\ell_{\max}}_i=-\frac{1}{2} \kappa_0\sum_{k=1}^N\sum_{j=1}^N  r_k^2r_j^2\int_{\mathbb{S}^2}\int_{\mathbb{S}^2} \nu^{\ell_{\max}}_k\big(\bold{x}_k+r_k\bold{t}\big) \nu^{\ell_{\max}}_j\big(\bold{x}_j+r_j\bold{s}\big)\\
		\cdot\bigg(\nabla_{\bold{x}_i}\frac{1}{\vert \bold{x}_j+r_j\bold{s}-(\bold{x}_k+r_k\bold{t})\vert}\bigg)\, d\bold{s}d\bold{t}.
		\end{multline*}
		
		A straightforward calculation shows that the only non-zero terms in this double sum involve $j\neq i, k=i$ and $j=i, k\neq i$. Consequently, we can write
		\begin{align*}
		\widetilde{\boldsymbol{F}}^{\ell_{\max}}_i=&-\frac{1}{2} \kappa_0r_i^2\sum_{\substack{k=1\\ k \neq i}}^N r_k^2\int_{\mathbb{S}^2}\int_{\mathbb{S}^2} \nu^{\ell_{\max}}_k\big(\bold{x}_k+r_k\bold{t}\big) \nu^{\ell_{\max}}_i\big(\bold{x}_i+r_i\bold{s}\big) \hphantom{\vert \bold{x}_i+r_i\bold{s}-(\bold{x}_k+r_k\bold{t})\vert}\\
		&\hphantom{\sum_{\substack{k=1\\ k \neq i}}^N \int_{\mathbb{S}^2}\int_{\mathbb{S}^2} \nu^{\ell_{\max}}_k\big(\bold{x}_k+r_k\bold{t}\big)\nu^{\ell_{\max}}_k\big(r_k\bold{t}\big)}\cdot\bigg(\nabla_{\bold{x}_i}\frac{1}{\vert \bold{x}_i+r_i\bold{s}-(\bold{x}_k+r_k\bold{t})\vert}\bigg)\, d\bold{s}d\bold{t}\\
		&-\frac{1}{2} \kappa_0r_i^2\sum_{\substack{j=1\\ j \neq i}}^N r_j^2\int_{\mathbb{S}^2}\int_{\mathbb{S}^2} \nu^{\ell_{\max}}_i\big(\bold{x}_i+r_i\bold{t}\big)\nu^{\ell_{\max}}_j\big(\bold{x}_j+r_j\bold{s}\big)\\&\hphantom{\sum_{\substack{k=1\\ k \neq i}}^N \int_{\mathbb{S}^2}\int_{\mathbb{S}^2} \nu^{\ell_{\max}}_k\big(\bold{x}_k+r_k\bold{t}\big)\nu^{\ell_{\max}}_k\big(r_k\bold{t}\big)}\cdot\bigg(\nabla_{\bold{x}_i}\frac{1}{\vert \bold{x}_j+r_j\bold{s}-(\bold{x}_i+r_i\bold{t})\vert}\bigg)\, d\bold{s}d\bold{t}.
		\end{align*}
		
		We can now use simple calculus and the symmetries in the above sum to obtain
		\begin{multline*}
		\widetilde{\boldsymbol{F}}^{\ell_{\max}}_i = -\kappa_0r_i^2\int_{\mathbb{S}^2} \nu^{\ell_{\max}}_i\big(\bold{x}_i+r_i\bold{t}\big)\sum_{\substack{j=1\\ j\neq i}}^N r_j^2\int_{\mathbb{S}^2} \nu^{\ell_{\max}}_j\big(\bold{x}_j+r_j\bold{s}\big) \bigg(\frac{\bold{x}_i+r_i\bold{s}-(\bold{x}_j+r_j\bold{t})}{\vert \bold{x}_j+r_j\bold{s}-(\bold{x}_k+r_k\bold{t})\vert^3}\bigg)\, d\bold{s} d\bold{t},
		\end{multline*}
		and therefore,
		\begin{align*}
		\widetilde{\boldsymbol{F}}^{\ell_{\max}}_i &= -\kappa_0\int_{\partial \Omega_i} \nu^{\ell_{\max}}_i(\bold{y})\sum_{\substack{j=1\\ j\neq i}}^N \int_{\partial \Omega_j} \nu^{\ell_{\max}}_j(\bold{x})\frac{\bold{y}-\bold{x}}{\vert \bold{y}-\bold{x}\vert^3}\, d\bold{x} d\bold{y}\\[0.4em]
		&= \kappa_0\int_{\partial \Omega_i} \nu^{\ell_{\max}}_i(\bold{y})\boldsymbol{E}_i(\bold{y})\, d\bold{y}= \boldsymbol{F}^{\ell_{\max}}_i.
		\end{align*}
	\end{proof}
	
	We conclude this discussion by observing that due to Theorem \ref{thm:equivalence}, all remarks concerning the electrostatic energy-based definition of the forces are equally applicable to the electric field-based definition of the forces. In particular, we can view the approximate forces $\{\boldsymbol{F}_i^{\ell_{\max}}\}_{i=1}^N$ as the gradient of the discrete potential energy surface (dPES) at the point $\{\bold{x}_i\}_{i=1}^N$.}
\end{document}